\documentclass[12pt]{amsart}

\usepackage[usenames,dvipsnames]{xcolor}

\usepackage{fancyhdr}
\usepackage{amsmath,amsthm,amsfonts,graphicx,amssymb,amscd}
\usepackage[all,cmtip]{xy}
\usepackage{graphicx, overpic, float}
\usepackage[mathscr]{euscript}
\usepackage{hyperref}
\usepackage{enumerate}
\usepackage{mathpazo}

\newtheorem{thm}{Theorem}[section]

\newtheorem{cor}[thm]{Corollary}
\newtheorem{lem}[thm]{Lemma}

\newtheorem{prop}[thm]{Proposition}

\theoremstyle{definition}
\newtheorem{defn}[thm]{Definition}

\newcommand{\T}{\mathscr{T}}

\renewcommand{\P}{\mathscr{P}}
\newcommand{\PP}{\mathcal{P}}
\newcommand{\CC}{\mathcal{C}}

\newcommand{\Q}{\mathcal{Q}}

\newcommand{\C}{\mathbb{C}}
\newcommand{\D}{\mathbb{D}}

\renewcommand{\H}{\mathbb{H}}

\newcommand{\ep}{\epsilon}

\usepackage{fancybox}

\usepackage{ulem}

\renewcommand{\emph}[1]{{\it #1}}

\usepackage{color}

\title{Harmonic maps and wild Teichm\"{u}ller spaces}

\author{Subhojoy Gupta}

\address{Department of Mathematics, Indian Institute of Science, Bangalore 560012, India.} 

\email{subhojoy@math.iisc.ernet.in}

\begin{document}
\setcounter{tocdepth}{4}
\maketitle


\begin{abstract} We use meromorphic quadratic differentials with higher order poles to parametrize the Teichm\"{u}ller space of crowned hyperbolic surfaces.  Such a surface is obtained on uniformizing a compact Riemann surface with marked points on its boundary components, and has non-compact ends with boundary cusps. This extends Wolf's parametrization of  the Teichm\"{u}ller space of a closed surface using holomorphic quadratic differentials. Our proof involves showing the existence of a harmonic map from a punctured Riemann surface to a crowned hyperbolic surface, with prescribed principal parts of its Hopf differential which determine the geometry of the map near the punctures.
\end{abstract}

\section{Introduction}

The Teichm\"{u}ller space $\T$ of a closed surface $S_g$ of genus $g\geq 2$ is the space of marked conformal or hyperbolic structures on $S_g$, and admits parametrizations that reflect various aspects (metric, complex-analytic, symplectic) of its geometry.  Many of these naturally extend to the case when the surface has boundaries and punctures, in which case the uniformizing hyperbolic metrics are such that the boundaries are totally geodesic, and the punctures are cusps. 

This paper deals with the case when each boundary component has an {additional} ``decoration", namely,  finitely many distinguished points.  These arise naturally in various contexts involving the arc complex (see for example \cite{Kaufmann-Penner}), and such a surface can be uniformized to a non-compact hyperbolic surface with \textit{crown ends} (Definition \ref{pend}) where the distinguished boundary points become ``boundary cusps" (see \cite{Chekhov}). 

The associated Teichm\"{u}ller spaces, and their symplectic and algebraic structures have been studied  before (see \cite{Chekhov}, \cite{YHuang}, \cite{Penner}).
In this paper, we provide  a parametrization of the Teichm\"{u}ller space of such crowned hyperbolic surfaces using meromorphic quadratic differentials with higher order poles.  
Holomorphic quadratic differentials and their geometry play a crucial role in Teichm\"{u}ller theory, and our result provides a link with this analytical aspect of the subject. \\

In the case of a closed surface, the work of M. Wolf (\cite{Wolf0}) parametrized Teichm\"{u}ller space $\T$ by the vector space $\Q$ of holomorphic quadratic differentials on any fixed Riemann surface $X$. 

Namely, Wolf proved that we have a homeomorphism 
\begin{equation*}
\Psi: \T \to \Q
\end{equation*}
that assigns to a hyperbolic surface $Y$ the Hopf differential of the unique harmonic diffeomorphism $h:X\to Y$ that preserves the marking. 

This turns out to be equivalent to Hitchin's parametrization of $\T$ in \cite{Hitchin} using self-dual connections on $X$ (see \cite{Nag}); both these approaches rely on a fundamental existence results for certain non-linear PDE (see \S2.3 and Theorem \ref{fund-thm}) that crucially depend on the compactness of the surfaces. \\

In this paper, we shall establish what can be described as a ``wild" analogue of the above correspondence - see \cite{Sabbah} for the far more general context of what is called the ``wild Kobayashi-Hitchin correspondence". 

That is,  $X$  shall  be a Riemann surface \textit{with punctures}, and the target for the harmonic maps would be the \textit{non-compact} crowned hyperbolic surfaces introduced earlier. It turns out that the Hopf differential $q$ of such a map has poles of higher orders (greater than two) at the punctures; if the order of the pole is  $n\geq 3$,  the number of boundary cusps of the corresponding crown end is $n-2$. Further, the analytic residue of the differential at the pole determines the ``metric residue" of the crown (see Definition \ref{metres}). This residue is part of the data of a \textit{principal part} $\textit{Pr}(q)$ that is a meromorphic $1$-form comprising the negative powers of $\sqrt q$ in a Laurent expansion with respect to a choice of coordinate chart  $U$ around the pole.  (See \S2 for details.) \\

Our key theorem is an existence result for such harmonic maps:



\begin{thm}[Existence]\label{thm1} Let $(X,p_1,\ldots p_k)$ be a closed Riemann surface with a non-empty set of marked points $D = \{p_1,\ldots p_k\}$ with fixed coordinate disks around each.  For $i=1,2,\ldots k$ let $n_i\geq 3$, and  let $Y$ be a crowned hyperbolic surface having crowns $\CC_i$ with $(n_i-2)$ vertices, together with a homeomorphism (or a marking) $f:X\setminus D \to Y$.

Then there exists a harmonic diffeomorphism $$h:X\setminus D \to Y$$ that is homotopic to $f$, taking a neighborhood of each $p_i$ to $\CC_i$. 


Moreover, such a map is unique if one prescribes, in addition, the principal parts at the poles,  having  residues  compatible with the metric residues of the corresponding crowns.

\end{thm}

(See Definition \ref{compat} for the notion of compatible residues.)

As a consequence, we establish the following extension of the Wolf-Hitchin parametrization:

\begin{thm}[Parametrization]\label{thm2} Let $X$ be a closed Riemann surface with a set of marked points $D$ with coordinate disks around them as above.  For a collection of principal parts $\mathscr{P}= \{P_1,P_2,\ldots P_k\}$ having poles of orders $n_i\geq 3$ for $i=1,2,\ldots k$,  let 
\begin{itemize}

\item $Q(X, D, \mathscr{P})$ 
be the space of meromorphic quadratic differentials on $X$  with principal part ${P}_i $ at $p_i$, and 

\item $\T(\mathscr{P})$
be the space of marked {crowned} hyperbolic surfaces homeomorphic to $X\setminus D$ , with $k$ crowns, each having $(n_i-2)$ boundary  cusps, with metric residues compatible with the residues of the principal parts $P_i$.
 
 \end{itemize}
 
  Then we have a homeomorphism
\begin{equation*}
\widehat{\Psi}:   \T(\mathscr{P}) \to Q(X, D, \mathscr{P}) 
\end{equation*}

that assigns, to any marked crowned hyperbolic surface $Y$, the Hopf differential of the unique harmonic map from $X$ to $Y$ with prescribed principal parts.
\end{thm}

\textit{Remark.} Throughout this article, we shall assume that the punctured surface $X \setminus D$ has negative Euler characteristic, though this is not strictly necessary. Indeed, we shall discuss the case of ideal polygons (genus zero and one puncture) in \S3.3.\\

We briefly recount some previous related work:

  In the case of meromorphic quadratic differentials with poles of order one, the appropriate Teichm\"{u}ller space is that of \textit{cusped} hyperbolic surfaces; the analogue of Wolf's parametrization was established in \cite{Lohk}; see also \S8.1 of the recent work \cite{Oscar-et-al}.
 
  When the hyperbolic surface in the target has geodesic boundary, the analogue of the existence result (Theorem \ref{thm1})  is implicit in the work of Wolf in \cite{Wolf3}. There, he proves the existence of harmonic maps from a noded Riemann surface to hyperbolic surfaces obtained by ``opening" the node; the Hopf differential of such a map has a pole of order two at either side of the node. 
 
 The existence of harmonic maps from a punctured Riemann surface to ideal polygonal regions in the Poincar\'{e} disk was recently shown by A. Huang in \cite{AHuang}. This was the first major advance since the work \cite{Au-Wan}, \cite{Au-Tam-Wan}, \cite{HTTW}, \cite{TamWan} which dealt with the case when the domain surface was the complex plane $\C$, and the Hopf differentials were polynomial quadratic differentials (see Theorem \ref{image}). We point out a description of the subspace of  polynomial differentials that correspond to a \textit{fixed} ideal polygonal image - see Proposition \ref{prop-pdiff}. 
 
 Furthermore, the work in \cite{GW2} established a generalization of the Hubbard-Masur Theorem by  proving the existence of harmonic maps to leaf-spaces of measured foliations with pole singularities; as in Theorem \ref{thm1}, such a harmonic map is unique if one further prescribes a principal part that is ``compatible" with the foliation. In analogy with the case of a  compact surface in the work of Wolf, we expect that that such foliations form the analogue of the Thurston compactification for the ``wild" Teichm\"{u}ller space that we describe in this paper, and our main result in \cite{GW2} is  a ``limiting" case of the results in this paper. We hope to pursue this in later work. \\
 
 Another key inspiration for the present work is the seminal work of Sabbah (\cite{Sabb}) and Biquard-Boalch (\cite{BiqB}) who prove the existence of the corresponding harmonic metrics for self-dual $\text{SL}_n(\C)$-connections with irregular singularities, that lead to a ``wild" Kobayashi-Hitchin correspondence for that complex Lie-group.
The work in this paper pertains, instead, to the ``wild" theory for the real Lie group $\text{SL}_2(\mathbb{R})$; implicit in this paper is a geometric understanding of the associated equivariant harmonic maps to the symmetric space, which in this case is the Poincar\'{e} disk.
It should be possible to generalize the methods in this paper to the case of other real Lie groups of higher rank (for the case of \textit{tame} singularities see the recent work of \cite{Oscar-et-al}); and in particular explore the structure of the spaces of ``wild" geometric structures whose monodromies would lie in the corresponding wild character variety.  Our hope is that establishing an analytic parametrization of such spaces will shed light on compactifications of the usual ``tame" spaces.\\


  

\textit{Strategy of the proof.} 
Existence results for harmonic maps to \textit{non-compact} targets is, in general, difficult as one needs to prevent the escape of energy out an end. Moreover, the traditional methods of existence results include heat-flow or sub-convergence of a minimizing sequence which crucially use finite energy (see, for example, \cite{LiTam}).
The key difficulty in proving the existence result in Theorem \ref{thm1} is the fact that the desired harmonic map has infinite energy as its Hopf differential has higher order poles. 

We start by proving an existence result for harmonic maps from a punctured disk  to hyperbolic crowns (see Theorem \ref{asm}) having Hopf differentials with prescribed principle parts. This is crucial to the subsequent argument as these maps serve as models of the asymptotic behavior of the desired map on the punctured Riemann surface. 
Our proof of this Asymptotic Models Theorem relies heavily on the technical results for the work concerning harmonic maps from $\C$ to ideal polygons that we mentioned above; in particular, on estimates on the geometry of the harmonic map that is inspired by the work of Wolf (\cite{Wolf3}) and Minsky (\cite{Minsky}).

The strategy of the proof of Theorem \ref{thm1} then is to consider an energy minimizing sequence defined on a compact exhaustion of the punctured surface $X\setminus D$, restricting to a desired model map around the punctures. An important step to ensure convergence is to provide a uniform energy estimate. Our previous work  in \cite{GW2}, required a ``symmetry" of the model map, that was a crucial to obtain a spectral decay of harmonic functions along a cylinder.   Here, a new argument using  ideas of A. Huang (see the Doubling Lemma in \S4.2), dispenses with the need for this symmetry.
Moreover, our previous work involved equivariant harmonic maps to certain augmented $\mathbb{R}$-trees, here the target is two-dimensional; this necessitates various modifications. \\

\textbf{Acknowledgments.} The idea of the project arose from a conversation with Yi Huang at a conference at the Chern Institute of Mathematics at Tianjin, and I thank him for his interest, and the conference organizers for their invitation. It is a pleasure to thank J\'{e}r\'{e}my Toulisse for his comments on a previous version of this article, Andrew Huang for conversations, and Michael Wolf, as some of these ideas emerged from a collaboration with him on a related project.  I also acknowledge the support by the center of excellence grant 'Center for Quantum Geometry of Moduli Spaces' from the Danish National Research Foundation (DNRF95) - this work was completed during a stay at QGM at Aarhus, and I am grateful for their hospitality.  The visit was supported by a Marie Curie International Research Staff Exchange Scheme Fellowship within the 7th European Union Framework Programme (FP7/2007-2013) under grant agreement no. 612534, project MODULI - Indo European Collaboration on Moduli Spaces.

\section{Preliminaries}

\subsection{Principal part}

We begin by defining the relevant space of quadratic differentials on the punctured disk that we shall use later:

\begin{defn}\label{qd-def} Let $n\geq 3$ and let $Q(n) $ be the space of meromorphic quadratic differentials on $\mathbb{D}^\ast$. Any such differential has a representative of the form:
\begin{equation}\label{qform}
\left(\frac{a_n}{z^n} + \frac{a_{n-1}}{z^{n-1}} + \cdots + \frac{a_3}{z^3} + \frac{a_2}{z^2} \right)dz^2
\end{equation}
where the coefficients are complex numbers.





\end{defn}

The fundamental complex invariant one can associate with such a differential is:

\begin{defn}[Analytic Residue]\label{ares}
The \textit{residue} of a meromorphic differential ($1$-form) $w$ at a pole $p$ is the integral $$ \text{Res}_p(q) = \pm \displaystyle\int\limits_\gamma w$$ where $\gamma$ is a loop around $p$. (Note that this is independent of the choice of such a loop).  We shall ignore the ambiguity of sign.

For a meromorphic quadratic differential $q$, the \textit{(quadratic) residue} shall be the residue of a choice of square root $\sqrt q$, that can be defined locally around the pole $p$.
\end{defn}

\bigskip
However in this paper we shall need the finer information of a \textit{principal part}, that we shall soon define.

\begin{defn}[Principal differential]\label{princ-def}
A \textit{principal (quadratic) differential} on $\mathbb{D}^\ast$ of order $n\geq 3$ is  defined to be a holomorphic quadratic differential of the form 
\begin{equation}\label{pp-eq}
z^{-\ep} \left(\frac{\alpha_r}{z^r} + \frac{\alpha_{r-1}}{z^{r-1}} + \cdots +  \frac{\alpha_1}{z} \right)^2dz^2
\end{equation}
 on $\mathbb{D}^\ast$ where $r = \lfloor n/2 \rfloor$ and $\ep = 0$ if $n$ is even and $1$ if $n$ is odd. 
 The space  of such principal differentials shall be  $$\mathrm{Princ}(n) \cong \mathbb{C}^{r-1}\times \mathbb{C}^\ast$$  via the map to the coefficients $(\alpha_1,\alpha_2,\ldots, \alpha_{r})$.
 \end{defn}

\begin{lem}\label{princ-lem}
Given a meromorphic  quadratic differential $q$ on $\mathbb{D}^\ast$ with a pole of order $n\geq 3$ at the origin, there is a unique principal differential  $P \in \mathrm{Princ}(n)$ such that $q - P$ is a holomorphic quadratic differential on $\mathbb{D}$.
\end{lem}
\begin{proof}


Suppose the quadratic differential $q$ is of the form (\ref{qform}) up to an addition of a holomorphic quadratic differential on $\mathbb{D}$.

Let $r = \lfloor n/2 \rfloor$.

Then on completing the square in the expression \eqref{pp-eq}  and comparing coefficients, we have:
$a_n = \alpha_r^2$, $a_{n-1}^2 = 2\alpha_r\alpha_{r-1}$, $a_{n-2} = \alpha_{r-1}^2 + 2 \alpha_{r} \alpha_{r-2}$, and in general $$a_{n-i}  = (\text{a quadratic expression involving }\alpha_{r-1}, \alpha_{r-2}, \ldots \alpha_{r-i+1}) + 2\alpha_r\alpha_{r-i}$$
for each $3\leq i\leq r-1$. 
Inductively,  $a_n$ determines $\alpha_r$, $a_{n-1}$ determines $\{\alpha_r,\alpha_{r-1}\}$, and  $a_{n-i}$ determines  $\{\alpha_r,\ldots \alpha_{r-i}\}$.
Thus the map $$(a_n,a_{n-1},\ldots a_{n-r+1}) \mapsto (\alpha_r,\alpha_{r-1},\ldots \alpha_1)$$ is well-defined and injective,
and thus $q$ determines $P$ such that the lemma holds.
\end{proof}

\textit{Remark.} For a $q,P$ as in the above Lemma, we shall sometimes say: 
\begin{equation*}
\sqrt q  = z^{-\ep} \left(\frac{\alpha_r}{z^r} + \frac{\alpha_{r-1}}{z^{r-1}} + \cdots +  \frac{\alpha_1}{z} \right)dz + \textit{a holomorphic differential}
\end{equation*}
although this only makes sense formally (as it depends on a suitable choice of a branch of the square root).  

\medskip
Using this lemma, we can finally define:

\begin{defn}[Principal part]
Given a meromorphic quadratic differential $q$ on $X$ with a pole of order $n\geq 3$ at $p$, and a choice of coordinate disk $U\cong \mathbb{D}$ around the point $p$, we define its \textit{principal part} $\text{Pr}(q)$  (relative to the choice of $U$)   to be the unique element of $\text{Princ}(n)$ such that $ q\vert_U - \text{Pr}(q)$ is a holomorphic quadratic differential on $U$.
\end{defn}



Recalling the definition of the space of quadratic differentials (Definition \ref{qd-def}), we have:

\begin{lem}\label{dims}
Let $n\geq 3$ and fix a principal differential $P \in \text{Princ}(n)$. Then the subspace $Q(P,n)$ of the quadratic differentials in $Q(n)$ with principal part $P$ is homeomorphic to $\mathbb{R}^{n-1}$ if $n$ is odd, and $\mathbb{R}^{n-2}$ if $n$ is even.
\end{lem}

\begin{proof}
As a consequence of the computations in the proof of Lemma \ref{princ-lem}, we have $P$ is uniquely determined by the $r$ coefficients $a_n, a_{n-1}, a_{n-2}, \ldots a_{n- r+1}$ where $r = \lfloor n/2 \rfloor$.  The remaining coefficients $a_2,a_3,\ldots a_{n-r}$ are a set of $(n-r-1)$ complex numbers, which contributes  $(n-2)$ real parameters if $n$ is even, and $(n-1)$ real parameters if $n$ is odd.
\end{proof}



\subsection{Hyperbolic crowns}

Throughout this paper, $(\mathbb{D}, \rho)$ will be the  Poincar\'{e} disk, or equivalently, the hyperbolic plane with the hyperbolic metric $\rho$. The hyperbolic distance function will be denoted by $d_\rho(\cdot, \cdot)$.

\begin{defn}[Crown, polygonal end]\label{pend} A \textit{crown}  $\CC$ with $m\geq 1$ ``boundary cusps"  is an incomplete hyperbolic surface bounded by a closed geodesic boundary $c$, and  a \textit{crown end} comprising bi-infinite geodesics $\{\gamma_i\}_{1\leq i\leq m}$ arranged in a cyclic order, such that the right half-line of the geodesic $\gamma_i$ is asymptotic to the left half-line of geodesic $\gamma_{i+1}$, where $\gamma_{m+1} =\gamma_1$.  A crown comes equipped with a labelling of the boundary cusps inbetween adjacent geodesic lines; the labels are $\{1,2,\ldots m\}$ in a cyclic order. 

A \textit{polygonal end} $\PP$ of a crown is the $\Gamma$-invariant bi-infinite chain of geodesic lines in $(\mathbb{D}, \rho)$ obtained by lifting the  cyclically ordered collection of geodesics $\{\gamma_i\}_{1\leq i\leq m}$ in $\CC$ to its universal cover, where $\Gamma =  \mathbb{Z}$ is the group generated by the hyperbolic translation corresponding to the geodesic boundary $c$.  A post-composition by an isometry ensures that the lift of $c$ is the bi-infinite geodesic $\alpha$ with endpoints at $\pm 1 \in \partial \D$, and that the cusp labelled ``1" is at $i\in \partial \D$; this normalization specified the polygonal end associated with the crown uniquely. 

\end{defn}

\textit{Remark.} The ideal endpoints of the chain of geodesics of the polygonal end limit to two points $p_\pm \in \partial \mathbb{D}$; the axis $\alpha$ of the lift of $c$ is between these two limit points. Thus, a polygonal end together with the data of the hyperbolic translation that it is invariant under, determines the crown $\CC$ by taking a quotient. 
See Figure 1. \\

\begin{figure}
  \centering
  \includegraphics[scale=0.45]{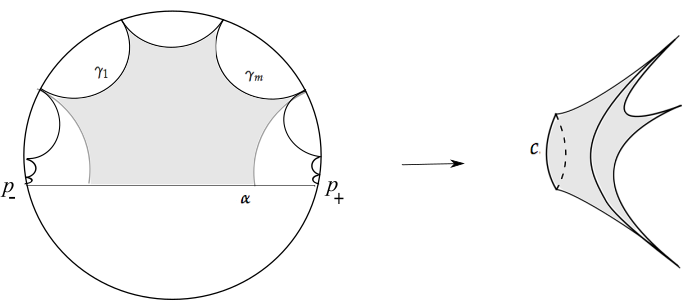}\\
  \caption{A polygonal end (left) is obtained as the universal cover of a hyperbolic crown (right).}
\end{figure}

The following associated notions will be useful later:

\begin{defn}[Truncation]\label{trunc} A \textit{truncation} of a crown $\CC$ with $m$ boundary cusps is obtained by removing a choice of disjoint neighborhoods $H_1,H_2,\ldots H_{m}$ at each ideal vertex of a $\CC$,  to obtain a subset $\CC \cap (H_1 \cup H_2 \cup \cdots H_{m})$ that is convex in the hyperbolic metric. 
The \textit{height} of the truncation of a cusp is the distance of the boundary of the deleted neighborhood, from the boundary of a maximal neighborhood of the cusp. 
\end{defn}

\textit{Remark.} The convexity of the truncated crown would be useful later (see the remark following Theorem \ref{fund-thm}). This means we need to truncate along geodesic arcs across the cusps, rather than horocycles. 

\begin{defn}[Metric residue]\label{metres} The \text{metric residue} of the hyperbolic crown $\CC$ with $m$ boundary cusps is defined to be $0$ when $m$ is odd, and equal to the alternating sum of lengths of geodesic sides of a truncation, when $m$ is even; like the analytic residue (see Den. \ref{ares}), we ignore the ambiguity of sign.
\end{defn}

It is easy to see  that:

\begin{lem} For  even $m$ the metric residue is well-defined, that is, it is independent of the choice of truncation.
\end{lem}
\begin{proof}
Truncating at a different height changes the length of the geodesic sides associated to that truncated sum by the same amount. However, in the alternating sum, the lengths of these sides appear with opposite signs, and hence the metric residue remains unchanged.
\end{proof}

The final notion we need for our crowns is:

\begin{defn}[Boundary twist] The hyperbolic crowns we shall consider shall come with an additional real parameter, the \textit{boundary twist}, that we associate with the geodesic boundary.  In the corresponding polygonal end in the universal cover, this can be thought of as the choice of a marked point on the bi-infinite line $\alpha$  that is the lift of the geodesic boundary; the parameter is then the signed distance of this point, from the foot of the orthogonal arc from the cusp labelled ``1" to $\alpha$.
\end{defn}

We can now define the relevant spaces of geometric objects:

\begin{defn} [Spaces of crowns]\label{poly-def} For $m\geq 1$, the space of hyperbolic crowns with $m$ labelled boundary cusps and a boundary twist
is denoted by $\mathsf{Poly}(m)$.  For any $a\in \mathbb{R}$, the subspace ${\mathsf{Poly}_a}(m)$ will comprise crowns with the specified metric residue $a$.
\end{defn}

\begin{lem}[Parametrizing the spaces]\label{poly1}
For $m\geq 1$, the space of polygonal ends ${\mathsf{Poly}}(m)\cong \mathbb{R}^{m+1}$, and  the space ${\mathsf{Poly}_a}(m)$ of those having metric residue equal to $a$, is homeomorphic to $\mathbb{R}^{m}$.
\end{lem}

\begin{proof}
Consider a crown $\CC$ with $m$ boundary cusps. As described earlier, its universal cover $\tilde{\CC}$ can be thought of as  the region in $(\mathbb{D}, \rho)$ bounded between a geodesic line $\alpha$ and a chain of geodesics invariant under the hyperbolic translation $\gamma$. One can normalize such that $\alpha$ is along an axis with end points $\pm 1 \in \partial \mathbb{D}$, and one of the lifts of a boundary cusp has an ideal vertex at $i \in \partial \mathbb{H}^2$.  A fundamental domain of the $\mathbb{Z}$-action is then uniquely determined by the $m$ points on the ideal boundary that are the endpoints of the remaining boundary cusps. These $m$ real parameters thus determine the hyperbolic crown $\CC$ with labelled cusps. 
The final real parameter is the choice of the basepoint which is anywhere along the axis $\alpha$, which is the boundary twist. 

For a fixed residue $a$, in the universal cover, a fundamental domain is in fact determined by the positions of $m-1$ ideal vertices (where one is already fixed at $i$ by the usual normalization).This is because the remaining ideal vertex is then uniquely determined from the metric residue, by the following hyperbolic geometry fact: If you fix two disjoint horodisks $H_l$ and $H_r$,  and let $a \in \mathbb{R}$, then there is a unique choice of an ideal point $x$ such that any horodisk $H_m$  based at $x$ satisfies $d_{\rho}(H_l,H_m) - d_{\rho}(H_r,H_m) = a$.
We leave the verification of this to the reader.
\end{proof}

\subsection{Crowned hyperbolic surface}

Throughout this paper, let $S$ be a compact orientable surface of genus $g\geq 1$ and $k\geq 1$ boundary components.


\begin{defn}\label{chs}
A \textit{crowned hyperbolic surface} is obtained by attaching crowns to a compact hyperbolic surface with geodesic boundaries by isometries along their closed boundaries. This results in an incomplete hyperbolic metric of finite area on the surface. Topologically, the underlying surface is $S\setminus \mathsf{P}$, where $\mathsf{P}$ is a set of finitely many points on each boundary component. 
\end{defn}

\begin{figure}
  \centering
  \includegraphics[scale=0.35]{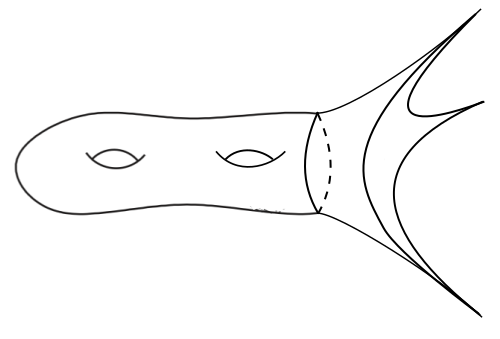}\\
  \caption{A crowned hyperbolic surface.}
\end{figure}

\textit{Remark.} On doubling a crowned hyperbolic surface, one obtains a cusped hyperbolic surface with an involutive symmetry. A compact Riemann surface with finitely many punctures admits a unique uniformizing metric that is hyperbolic and has cusps at the punctures. Thus a crowned hyperbolic surface can be thought of as the unique uniformizing metric for a compact Riemann surface with boundary, having  distinguished points on each boundary component.  






\begin{defn}[Teichm\"{u}ller space of crowned hyperbolic surfaces]\label{T-chs} Let 
$$\mathcal{T}(S, m_1,m_2,\ldots m_k) = \{ (X,f)  \text{ }  \vert \text{ } X \text{ is a marked hyperbolic surface with }k\text{ crowns}\}/\sim$$ 
such that the $i$-th crown has $m_i\geq 1 $ boundary cusps, for each $i=1,2,\ldots k$. 

Here, the marking $f:S\to X$ is a homeomorphism that takes each boundary component to the crown ends. Two marked surfaces are equivalent, that is, $(X,f) \sim (Y,g)$ if there is an isometry $i:X\to Y$ that is homotopic to $g\circ f^{-1}$ via a homotopy that keeps each boundary component fixed,  and $g\circ f^{-1}$ is a homeomorphism that does not Dehn-twist around any crown end. 
\end{defn}

We then have the following parametrization:

\begin{lem} Let $k\geq 1$ and fix integers $m_1,m_2,\ldots m_k \geq 1$.  The Teichm\"{u}ller space of crowned hyperbolic surfaces $\mathcal{T}(S, m_1,m_2,\ldots m_k)$ is homeomorphic to $\mathbb{R}^{\chi}$ where $\chi = 6g-6 +  \sum\limits_{i=1}^k (m_i + 3)$.
\end{lem}
\begin{proof}
It is well-known that a hyperbolic surface with $k$ geodesic boundary components  is determined by $6g-6 + 3k$ real parameters. This can be seen, for example, by considering Fenchel-Nielsen parameters on the bordered surface.

By Definition \ref{chs}, a  crowned hyperbolic surface is obtained by an isometric identification of a bordered hyperbolic surface and a collection of crowns along their geodesic boundaries. (See Figure 2.) 

By Lemma \ref{poly1} each hyperbolic crown  is determined by $m_i$ real parameters.  The additional  ``boundary twist" parameter determines the ``twist" in the identification of the crown boundary with that of the bordered hyperbolic surface. However,  the length of the geodesic boundary must match with that of the bordered hyperbolic surface to achieve the isometric identification. 

Thus, each crown adds $m_i$ real parameters to the crowned hyperbolic surface, and we obtain a total of $\chi$ parameters, as claimed. 
\end{proof}

\textit{Remark.} Our Definition \ref{T-chs} differs from that in Y.Huang  \cite{YHuang} or Chekhov-Mazzocco \cite{C-M}  in that the parameter spaces they define are ``decorated", and in particular also record the data of a choice of truncation  for each boundary cusp; this results in an additional $\sum\limits_{i=1}^k m_i$ real parameters in the notation above.  In any case, the above parameterization can also be derived from their work.\\

Finally, we note the following parametrization of marked crowned hyperbolic surfaces when we fix the metric residues for crowns with an even number of boundary cusps. The proof follows from that of Lemma \ref{poly1}, and we leave the details to the reader:

\begin{cor}\label{wtdim}  Let $\mathcal{I} \subset \{1,2,\ldots k\}$  be the subset of indices such that $m_i$ is even for $i\in \mathcal{I}$. Fix an ordered tuple of real numbers $A_\mathcal{I} = \{a_\iota\}_{\iota \in \mathcal{I}}$. Let $\mathcal{T}(A_\mathcal{I})$ be the subspace of $\mathcal{T}(S, m_1,m_2,\ldots m_k)$ comprising marked hyperbolic surfaces with the crown end $\CC_\iota$ having metric residue $m_\iota$, for each $\iota \in \mathcal{I}$. Then $\mathcal{T}(A_\mathcal{I})\cong \mathbb{R}^{\chi - \lvert \mathcal{I} \rvert }$ where $\chi = 6g-6 +  \sum\limits_{i=1}^k (m_i + 3)$. 
\end{cor}


\subsection{Harmonic maps between surfaces}

In this section we shall recall some basic facts about harmonic maps with negatively curved targets.

Let $(X, \sigma \lvert dz\rvert^2)$ and  $(Y, \rho \lvert dw\rvert^2)$  be  Riemann surfaces with a conformal metrics. 

Throughout this paper, $\rho$ shall be a hyperbolic metric, that is, has constant negative curvature $-1$.


\begin{defn}[Harmonic map] \label{hmap}

A {harmonic map} 
\begin{center}
$h: (X, \sigma \lvert dz\rvert^2) \to (Y, \rho \lvert dw\rvert^2)$ 
\end{center}
 is a critical point of the energy functional
\begin{equation*}
\mathcal{E}(f) = \int\limits_X  e(f) dzd\bar{z}
\end{equation*}
which is defined on all maps from $X$ to $Y$ with locally square-integrable derivatives.
Here,
\begin{equation}\label{energy}
e(f) = \lVert \partial h_z\rVert^2 + \lVert \partial h_{\bar{z}}\rVert^2
\end{equation}
is the \textit{energy density} of $f$, where
\begin{center}
$\lVert \partial h\rVert^2  = \rho^2 \lvert h_z\rvert^2$, $\lvert \bar{\partial} h\rvert^2  = \rho^2  \lVert h_{\bar{z}}\rVert^2$
\end{center}

Note that the energy depends only on the conformal class of the metric on the domain $X$; in what follows we shall often drop specifying the choice of such a metric.

The corresponding Euler-Lagrange equation that $h$ satisfies is:
\begin{equation}\label{har}
h_{z\bar{z}} +  (\ln \rho)_w h_z h_{\bar{z}}= 0
\end{equation}
where $z$ and $w$ are the local coordinates on $X$ and $Y$ respectively.

\end{defn}

For compact surfaces, we have the following fundamental existence result:

\begin{thm}[Eells-Sampson, Hartman, Al'ber, Sampson, Schoen-Yau]\label{fund-thm}
Let $(X, \sigma \lvert dw\rvert^2)$ and  $(Y, \rho \lvert dw\rvert^2)$  be compact Riemann surfaces with conformal metrics such that $Y$ is negatively curved. Then there exists a unique  harmonic diffeomorphism $h:X\to Y$ in the homotopy class of any diffeomorphism, that is a minimizer of the energy functional \eqref{energy}.
\end{thm}

\textit{Remark.} The work of Hamilton and Schoen-Yau also extends this to the case when the surfaces have boundary;  namely if the boundary of $Y$ is convex (alternatively, having non-negative geodesic curvature), there exists a harmonic map in the homotopy class as above, with any prescribed boundary map.  (See Theorem 4.1 of \cite{SchoenYau}, and also pg. 157-8 of \cite{Hamilton}.)



\subsection*{Prescribing the Hopf differential} 
\begin{defn}
The \textit{Hopf differential} of such a map is the quadratic differential given by the local expression
\begin{equation*}
\text{Hopf}(h)(z) = \phi(z) dz^2 :=   \rho(h(z))h_z\overline{h_{{z}}} dz^2 
\end{equation*}

and it is well-known that it is holomorphic if and only if $h$ is harmonic  (see for example, \cite{Sampson}). 
\end{defn}

The existence and uniqueness theorems for harmonic maps with prescribed Hopf differentials have been proven in much more general settings ( see \cite{TamWan}, \cite{HTTW}, \cite{Au-Wan2}). The statement we shall need, are given below:

\begin{thm}[Theorem 3.2 and Proposition 4.6 of \cite{TamWan}]\label{tamwan}  Let $q$ be any holomorphic quadratic differential on $\mathbb{C}$. Then there is a harmonic map  $$h:\mathbb{C} \to (\mathbb{D},\rho)$$ that  is a diffeomorphism to its image, and has Hopf differential $q$.

Moreover, if $h_1$ and $h_2$ are two such orientation-preserving harmonic diffeomorphisms, then $h_2 \circ h_1^{-1}$ is an isometry from $h_1(\mathbb{C})$ to $h_2(\mathbb{C})$. In fact, $h_2 = A\circ h_1$ for an isometry $A:(\mathbb{D},\rho) \to (\mathbb{D},\rho)$.
\end{thm}

In fact, as we shall see in the next subsection, one can deduce more about the image in terms of the quadratic differential $q$.


\subsection{Geometric estimates}


The estimates on image of the harmonic map are in terms of the metric induced by the Hopf differential defined below. The results in this section have been culled from the work of \cite{Han-Remarks}, \cite{Minsky}, \cite{Wolf2}, \cite{Au-Wan}, \cite{HTTW}, \cite{Au-Tam-Wan}.

\begin{defn}[$q$-metric] 
The metric induced by a quadratic differential $q$ defined on a Riemann surface $X$  (also referred to as the $q$-metric)  is a conformal metric given by the local expression $\lvert q(z) \rvert \lvert dz\rvert^2$, which is singular at the zeroes of the quadratic differential $\phi$.  The holomorphicity of $q$ then implies that the curvature vanishes away from these singularities, and hence the metric is a \textit{singular flat metric}.
\end{defn}

We also have the following notion:

\begin{figure}
  \centering
  \includegraphics[scale=0.5]{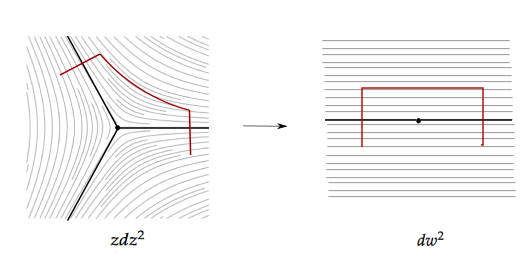}\\
  \caption{The horizontal foliation induced by a quadratic differential is acquired via the canonical chart. Horizontal and vertical segments (shown in red) have lengths in the $q$-metric that are the Euclidean distances in the $w$-plane.}
\end{figure}

\begin{defn}\label{h-v}
In any local chart there is a change of coordinates $z\mapsto w = x+iy$ such that the quadratic differential $\phi(z)dz^2$ transforms to $dw^2$. 
The ``{horizontal}" (\textit{resp.} ``{vertical}") direction is the $x$ -(\textit{resp}. $y$-) direction in these canonical charts. 
More intrinsically, they  can be defined to be the directions in which the quadratic differential takes real and positive  (\textit{resp.} real and negative) values.  
A differentiable arc on $X$ is said to be {horizontal} (\textit{resp.} {vertical}) if its tangent directions are so.  See Figure 3. 
\end{defn}

\textit{Example}. For $d\geq 0$,  the quadratic differential $z^ddz^2$ on $\mathbb{C}$ with a pole of order $d+4$ at infinity, the horizontal and vertical directions are obtained by pulling back the horizontal and vertical lines in $\mathbb{C}$. In particular, each ray $\gamma_j(t) = \{ te^{i2\pi \frac{j}{d+2}} \}$  for $1\leq j \leq d+2$ is  horizontal. By a change of coordinate $z\mapsto 1/z$, we see that these are the \textit{horizontal rays} at the pole at infinity. More generally, the meromorphic quadratic differential of the form \eqref{qform} can be thought of as a perturbation of this; its horizontal leaves are asymptotic to the above directions at the pole.\\

\textit{Geometric interpretation.} A computation shows that the pullback of the metric in the target satisfies: 
\begin{equation}\label{pullb}
h^\ast(\rho(u) \lvert du\rvert^2) =  (e+2) dx^2 + (e -2)dy^2
\end{equation}
where $e$ is the energy density of $h$ (as in Definition \ref{hmap}) with respect to the metric induced by $\phi$. See formula (3.6) in \cite{HTTW}.
Then the horizontal and {vertical directions} (in the $xy$-coordinate chart) are the directions of maximal and minimal stretch of the harmonic map.\\

We shall state the estimates  for a planar domain though the  estimates for a harmonic map  $h:X \to (Y, \rho)$  with Hopf differential $q$ work for any  Riemann surfaces $X$ and $Y$, under the assumptions that 
\begin{itemize}
\item the conformal metric $\rho$ on $Y$ has constant negative curvature $-1$, and 
\item the \textit{singular-flat $q$-metric} induced by the Hopf differential $q$ is \textit{complete}.
\end{itemize}

\textit{Remark.} The first requirement above is relaxed in the work in \cite{HTTW}, to encompass all Cartan-Hadamard spaces.  The second requirement, that is automatic in the case $X= \mathbb{C}$, is needed to apply the method of sub-and-super solution in the proof of the following estimate.

\begin{prop}[Horizontal and vertical segments]\label{prop-est} 
Let $h:\mathbb{C}\to (\mathbb{D},\rho)$ be a harmonic map that is a diffeomorphism to its image, with Hopf differential $q$.

Let $\gamma_{h}$ and $\gamma_{v}$ on $\mathbb{C}$  be horizontal and vertical segments in the $q$-metric , each of length $L$ and a distance $R>0$, in the $q$-metric,  from the singularities of $q$.

  Then their images $h(\gamma_h)$ and $h(\gamma_v)$ have lengths $2L + O\left(e^{-\alpha R}\right)$ and $O\left(L e^{-\alpha R}\right)$ respectively, where $\alpha>0$ is a universal constant. 
  
  Moreover, the image of $\gamma_{h}$ is an arc with geodesic curvature $O\left(e^{-\alpha R}\right)$. In particular, it is a distance $\ep(R)$ away from a geodesic segment, where $\ep(R) \to 0$ as $R\to \infty$.

\end{prop}

For the convenience of the reader, we provide a sketch of the proof; for details  we refer to Lemma 2.1 and Corollary 2.2 of \cite{Au-Wan}. A good overview of the methods can be obtained in \cite{Han-Remarks}. 

\begin{proof}[Sketch of the proof]
The basic analysis concerns the following Bochner equation for the harmonic map :
\begin{equation}\label{boch}
\Delta w =   e^{2w} - e^{-2w} \lvert q(z)\rvert^2  
\end{equation}
where $w= \ln \lVert \partial h \rVert$ and  $\Delta$ is the usual Laplacian. (See \cite{SchoenYau}, and the discussion in \S1 of \cite{HTTW}.)

Setting $$w_1(z) = w(z) - \frac{1}{2} \ln \lvert q(z)\rvert$$
and using the Laplacian $\Delta_q$ with respect to the $q$-metric, we obtain
\begin{equation}\label{boch2}
\Delta_q w_1 =   e^{2w_1} - e^{-2w_1}
\end{equation}
The technique of sub- and super-solutions then gives us the following decay estimate on solutions: 
For any $z\in X$ we have
\begin{equation}\label{est-sol}
0 \leq w_1(z)  \leq e^{-Cr(z)}
\end{equation}
where $r(z)$ is the distance from the singularity set  in the $\phi$-metric on $X$, and $C>0$ is an absolute constant. 

Moreover by a standard gradient estimate we obtain:
\begin{equation}\label{grad}
\lvert \nabla w_1 \rvert  = O(e^{-\alpha R})
\end{equation}
where the gradient is with respect to the $q$-metric.

From a simple calculation, the energy density $$e = 2\cosh (2w_1) \approx 2 \cdot \left( 1 + O(e^{-Cr(z)}) \right)$$  so the length estimates follow by considering \eqref{pullb}, namely
\begin{equation}\label{lhh}
L(h(\gamma_h)) = \displaystyle\int_0^L  \sqrt{e+2} dx  = 2L + O(e^{-\alpha R})
\end{equation}
\begin{equation}\label{lhv}
L(h(\gamma_v)) = \displaystyle\int_0^L  \sqrt{e-2} dy = O(Le^{-\alpha R}) 
\end{equation}
where we have used \eqref{est-sol}.

Moreover, by the formula (3.7) of \cite{HTTW}, the geodesic curvature

\begin{equation}\label{geodk}
\kappa(h(\gamma_h)) = -\frac{1}{2}(e-2)^{1/2}(e+2)^{-1} \frac{\partial e}{\partial y} = O(e^{-\alpha R})
\end{equation}
where the final estimate follows from \eqref{grad}.
\end{proof}

The proofs of Lemmas 3.2-4   of \cite{HTTW} show that:

\begin{prop}[Images of horizontal rays]\label{rays} Let $H = \{z(t)\}_{t\geq 0}$ be a horizontal ray in the $q$-metric on $\mathbb{C}$, and $h:\mathbb{C} \to (\mathbb{D},\rho)$ a harmonic map as before. Then the image of the ray is asymptotic to an ideal boundary point $\theta \in \partial \mathbb{D}$, that is, $\lim\limits_{t\to \infty} h(z(t)) = \theta$. 
Moreover, horizontal rays that are asymptotic to the same direction in $\mathbb{C}$ have images that limit to the same point in the ideal boundary, and different asymptotic directions give rise to distinct ideal points. 
\end{prop}

\begin{proof}[Sketch of the proof]
The first statement follows from the estimate of geodesic curvature \eqref{geodk}, and a hyperbolic geometry fact, as in Lemma 3.2 of \cite{HTTW}.
If they are asymptotic to the same direction in $\mathbb{C}$, then the vertical distance between them is uniformly bounded, and hence by \eqref{lhv} the $\rho$-distance between their images tends to zero.
For a pair of distinct asymptotic directions in $\mathbb{C}$, pick a sequence of horizontal leaves asymptotic to these directions, with increasing vertical height from a basepoint. 
By Proposition \ref{prop-est} the images of these limit to a bi-infinite geodesic line $\gamma$. 
To show that the limit points are distinct, it suffices to show that  $\gamma$ is a finite distance away from the basepoint, which in turn follows from Lemma 1.1 of \cite{Au-Tam-Wan} which implies that the image of the vertical segments (from the basepoint to the horizontal leaves) have finite length.
\end{proof}




\subsection*{Example: Polynomial quadratic differentials}
This  was studied in \cite{Wan} (see also \cite{Au-Wan}), and in \cite{HTTW}, where they showed:

\begin{thm}[Han-Tam-Treibergs-Wan]\label{image}
Let $n\geq 5$ and let 
\begin{equation}\label{polydiff}
 (z^d + a_{d-2}z^{d-2} + \cdots + a_1z + a_0) dz^2
 \end{equation}
  be a (polynomial) quadratic differential on $\mathbb{C}$, where the degree $d = n-4$ and coefficients $a_i \in \mathbb{C}$ for $0\leq i\leq d-2$. 
Then there exists a harmonic map 
\begin{equation*}
h:\mathbb{C}\to (\mathbb{D}, \rho)
\end{equation*}
 which is a diffeomorphism to an ideal polygon with $(n -2)$ vertices in the boundary at infinity, and whose Hopf differential is $q$. Moreover, the harmonic map is unique if three of the ideal vertices are prescribed to be $\pm 1, i$.
 \end{thm}
 
 \textit{Example.} The quadratic differential $(z^2 +a)dz^2$ for $a\in \mathbb{C}$,  has a pole of order $6$ at $\infty$, with residue $\sqrt{a}$ and the image of the harmonic map $h$ above is an ideal quadrilateral with a cross ratio, and metric residue,  determined by the imaginary part of $a$. See Figure 4, and  \cite{Au-Wan} for details.\\

In fact, as a consequence of the geometric estimates of \S2.5, we have the following picture.

We first define:

\begin{defn}[Polygonal exhaustion]\label{polyex} Given a meromorphic quadratic differential $q$  with a pole of order $n\geq 3$, 
a \textit{polygonal exhaustion} at the pole is a nested sequence of annular regions
$$P_1 \subset P_2 \subset \cdots P_j \subset \cdots$$

such that their union is a neighborhood of the pole, and each boundary $\partial P_j$ is an $(n-2)$-sided polygon with sides that are alternately horizontal and vertical segments.   (\textit{cf.} Figure 3.) 

\end{defn}

We first note the following link between the geometry of the $q$-metric induced by the quadratic differential, and the quadratic analytical residue at the pole  (see Definition \ref{ares}):

\begin{lem}[Residue and $q$-metric]\label{qres}  Let $q$ be a  meromorphic quadratic differential with a pole of even order $n\geq 4$, and let $\{P_j\}_{j\geq 1}$ be a polygonal exhaustion, as in the preceding definition, of a neighborhood $U$ of the pole.

Then for any $j\geq 1$ the alternating sum of the lengths of the horizontal sides of  $\partial P_j$  equals the real part of the (quadratic) residue $\alpha$  at the pole (up to sign).  That is, we have
\begin{equation}\label{rea}
\Re(\alpha) = \pm \sum\limits_{i=1}^{(n-2)/2} (-1)^i l_i.
\end{equation}
\end{lem}




\begin{proof}

The following argument  culled from the proof of Theorem 2.6 in \cite{Gup2}.

Recall from Defn \ref{ares} that the quadratic residue
\begin{equation}\label{asum}
\alpha =  \pm \displaystyle \int\limits_{\partial P_j} \sqrt q 
\end{equation}

since $\partial P_j$ is  a curve homotopic to the puncture. 

By a change of coordinates $z\mapsto \zeta$ that takes the quadratic differential to $d\zeta^2$, the integral over three successive edges (horizontal-vertical-horizontal) of $\partial P_j$ equals integrating the form $d\zeta$ on the $\zeta$-plane over a horizontal  edge that goes from the right to left on the upper half-plane followed by one over a vertical edge followed by one over a horizontal edge that goes from left to right in the lower half-plane. The integral in \eqref{asum} over the horizontal sides picks up the horizontal lengths that contributes to the real part of the integral, while the integral over the vertical side contributes to the imaginary part of the integral. 

Note that the sign switches  over the two successive horizontal edges, and hence we obtain \eqref{rea}. 
\end{proof}

\bigskip

\begin{figure}
  \centering
  \includegraphics[scale=0.5]{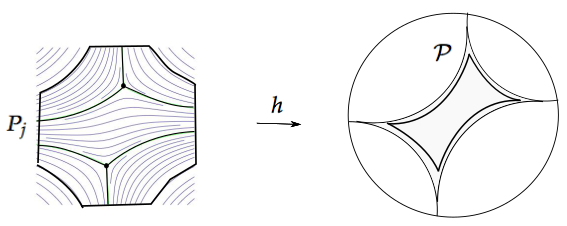}\\
  \caption{ The boundary of a polygonal region $P_j$ (left) comprises alternate horizontal and vertical segments. This maps close to a truncation of the ideal polygon $\PP$. See Proposition \ref{poly-ex}.}
\end{figure}

For the following Proposition, note that the \textit{metric residue} of an ideal polygon $\PP$ with an even number of sides can be defined like that in the case of a hyperbolic crown (Defn \ref{metres}): namely, it is the real number one obtains by taking an alternating sum of the hyperbolic lengths of the geodesic segments obtained by truncating $\PP$. (For an ideal polygon with an odd number of sides, the metric residue is defined to be zero.)

\begin{prop}[Asymptotic image]\label{poly-ex}
Let $q$ be a polynomial Hopf differential of a harmonic map $\hat{h}:\mathbb{C} \to \mathbb{D}$.
Let $\PP$ be the ideal polygon with $(n-2)$ ideal vertices that is the image of $h$ (see Theorem \ref{image}).

Construct a polygonal exhaustion $\{P_j\}_{j\geq 1}$ at the pole at infinity (see Definition \ref{polyex}) where each boundary $\partial P_j$ has horizontal and vertical sides of length $L_j \pm O(1)$  in the $q$-metric such that $L_j \to \infty$.

Then the image $h(P_j)$ is  $\ep$-close to the boundary of a truncation of $\PP$, where $\ep\to 0$ as $j \to \infty$.

Moreover, the metric residue of $\PP$ equals the twice the real part of the analytical  residue of $q$ at the pole at infinity. 
\end{prop}

(Throughout the paper, ``$O(1)$" shall denote a quantity bounded by a universal constant.) 

\begin{proof}
As the distance from the zeroes of $q$ increases, the horizontal sides of the polygonal boundaries are mapped closer to the geodesic sides of $\PP$, and the vertical sides contract to almost-horocyclic arcs in the cusps, by  Proposition \ref{prop-est}.

By the distance estimates (Prop. \ref{prop-est}) hold, and the hyperbolic length of the image of any horizontal side differs from twice its length in the $q$-metric by $\ep>0$ , where the error term $\ep$  depends on the distance of $\partial P_j$ from the zeroes of $q$, and hence can be made arbitrarily small by choosing $j$ sufficiently large.

Hence the metric residue $a$ of the ideal polygon (\textit{cf.} Defn.\ref{metres} and the discussion immediately preceding this proposition) differs from twice the alternating sum of the $q$-lengths  of the horizontal sides by $O(\ep)$.
   
However by Lemma \ref{qres}, this alternating sum of the lengths of the horizontal sides of $\partial P_j$ equals the real part of the quadratic residue at the pole.  Letting $\ep \to 0$, we have that this equals $a/2$, as required. 
 \end{proof}




\textit{Remark.} Such a picture also holds for harmonic maps to hyperbolic crowns $h:\mathbb{D}^\ast \to \CC$ having Hopf differentials with poles at the origin, as we shall exploit in \S3. Namely, one can construct a polygonal exhaustion at the pole such that their images are asymptotic to the end of the crown.\\

For future use, we shall use the following definition from the preceding discussion:

\begin{defn}[Compatibility of residue]\label{compat} A principal part $P \in \text{Princ}(n)$ where $n\geq 3$ is said to be \textit{compatible} with a hyperbolic crown $\CC \in \text{Poly}(n-2)$ (and vice versa)  if its analytical residue  equals half of the metric residue of the crown. (See Definitions \ref{ares} and \ref{metres} for these notions of residue.) 
\end{defn}

\section{Asymptotic model maps}

Before we prove the main theorem, we prove the existence of harmonic maps from $\mathbb{D}^\ast$ into a hyperbolic crown $\CC$ that will serve as ``model maps" defined on a neighborhood of the puncture on $X\setminus p$.




\subsection*{Doubling the domain} To be able to use Theorem \ref{tamwan} and the estimates of \S2.5 (which assume completeness of the domain), we shall 
consider quadratic differentials defined on the punctured complex plane $\mathbb{C}^\ast$ having an involutive symmetry, namely an invariance under $z\mapsto 1/z$.

\begin{defn}[Symmetrizing $q$] Let $n\geq 3$. For any quadratic differential $q\in \text{Q}(P,n)$ of the form \eqref{qform}, we define a differential  $q_{sym}$ on $\mathbb{C}^\ast$ as follows:
 \begin{equation}\label{hatq}
 {q}_{sym} =  \left( \frac{a_n}{z^n} + \frac{a_{n-1}}{z^{n-1}} + \cdots + \frac{a_3}{z^3} +  \frac{a_2}{z^2} +  \frac{b_{-1}}{z} + b_0 + b_1z + \cdots b_{n-4} z^{n-4}\right)dz^2
   \end{equation}
where   $b_i = a_{i+4} \text{ for } -1\leq i \leq n-4$.
Note that this is the unique choice of coefficients such that 
\begin{itemize}
\item $q_{sym}$ has the involutive symmetry $q_{sym}(z) = q_{sym}(1/z)$, and

\item $q_{sym}$ and $q$ have the same principal part at the pole at $0$.

\end{itemize}
\vspace{.1in}

Note that conversely, $q_{sym}$ determines $q$ uniquely. 

We shall consider the  space $Q_{sym}(P,n)$ of symmetric differentials arising this way, up to a scaling by positive reals. Since $P$ is determined by the coefficients $(a_n,\ldots, a_{n-r-1})$ where  $r = \lfloor n/2 \rfloor$ (see Lemma \ref{dims}),  and we have that the space $Q_{sym}(P,n)$  is homeomorphic to $Q(P,n)$.

\end{defn}



\medskip

The main result of this section is:

\begin{thm}[Asymptotic Models]\label{asm}
Let $n\geq 3$, and let $P \in \text{Princ}(n)$. For any $q_{sym} \in Q_{sym}(P,n)$ and a hyperbolic crown $\CC$, there exists a harmonic map $${h}: \mathbb{D}^\ast\to \CC$$ that has Hopf differential $q_{sym}\vert_{\mathbb{D}^\ast}$ and is asymptotic to the crown end with $(n-2)$ boundary cusps.
The asymptotic image is independent of rescaling the differential by a positive real, and this assignment 
defines a homeomorphism
\begin{equation}\label{phimap}
{\Phi}:  \mathrm{Q}_{sym}(P, n)   \to {\mathsf{Poly}_a}(n-2).
\end{equation} 
where $a$ is the real part of the residue of the principal part $P$. 
\end{thm}

(Recall that $\mathsf{Poly}_a(n-2)$ is the space of hyperbolic crowns with $(n-2)$ boundary cusps and fixed metric residue $a$.)\\

\textit{Remark.} Note that by Lemmas \ref{dims} and \ref{poly1} the dimensions of the spaces in \eqref{phimap} are identical. This is crucial, as the homeomorphism will be finally obtained as an application of the Invariance of Domain.


\subsection{Existence} 
We shall work in the universal cover.

In what follows , $\mathbb{H}^2$ is the upper half plane and $$\pi:\mathbb{H}^2 \to \mathbb{D}^\ast$$ is the universal covering map $\pi(w) = e^{2\pi iw}$, with the action of $\mathbb{Z}$ by deck-translations being $w\mapsto w +1$.\\

The main step towards proving Theorem \ref{asm} is:

\begin{prop}[Equivariant maps on $\H^2$]\label{exist}
For $n\geq 3$ fix a $P \in \text{Princ}(n)$. 
There exists a meromorphic quadratic differential $q$ on $\mathbb{D}$ with a pole of order $n$ at the origin and principal part $P$, and a harmonic map $$\tilde{h}:\mathbb{H}^2 \to (\mathbb{D},\rho)$$ such that
\begin{enumerate}
\item its Hopf differential is $\pi^\ast(q_{sym})$,
\item its image is a polygonal end in $\mathrm{Poly}_a(n-2)$, and
\item it is equivariant with respect to a $\mathbb{Z}$-action on domain and range, that is, 
\begin{equation}\label{T-act}
 \tilde{h}(w+1) = T\circ \tilde{h}(w)
 \end{equation}
 for a hyperbolic isometry $T$ of the Poincar\'{e} disk. 
\end{enumerate}

\end{prop}

Note that the equivariant harmonic map $\tilde{h}$ above on the universal cover of $\mathbb{D}^\ast$ descends to define a harmonic map to a hyperbolic crown.\\

Such an existence result for harmonic maps with prescribed Hopf differential is known when the domain is  $\mathbb{C}$ (see Theorem \ref{tamwan}), and the geometric estimates of  \S2.5 hold when the Hopf differential metric is  \textit{complete}.

We have already taken care of the second requirement by extending the quadratic differential $q$ on $\D^\ast$ to $q_{sym}$ defined on the punctured plane $\C^\ast$. 
To satisfy the first requirement, we consider the lift of the quadratic differential $q_{sym}$ to the universal cover, namely the differential  $\hat \pi^\ast(q_{sym})$ on $\C$ , where $\hat \pi:\mathbb{C} \to \mathbb{C}^\ast$ is the universal covering defined by $\hat \pi(w) = e^{2\pi iw}$.







By Theorem \ref{tamwan} we have the existence of such a harmonic map
\begin{equation}\label{mapc}
\hat{h}:\mathbb{C}\to (\mathbb{D},\rho)
\end{equation}
with Hopf differential $\hat{\pi}^\ast({q}_{sym})$, which is unique once one specifies a normalization of the image. 
 
  To complete the proof of Proposition \ref{exist}, we shall show that the restriction of $\hat{h}$ to the upper half plane
  \begin{equation}\label{maph}
 \tilde{h} = \hat{h}\vert_{\mathbb{H}^2} :\mathbb{H}^2\to (\mathbb{D},\rho)
 \end{equation}
 
  is the desired equivariant map to a polygonal end.
  
  For this, we need to show (a) the image is asymptotic to a bi-infinite chain of geodesics, and (b) the map is equivariant with respect to the $\mathbb{Z}$-action generated by translation $w\mapsto w+1$ in the domain and a hyperbolic translation in the target as in \eqref{T-act}.

   
 \subsection*{Determining the image}
We can use the analytical estimates (see \S2), since the pullback quadratic differential metric $\hat \pi^\ast(q_{sym})$ on $\mathbb{C}$ defines a metric that is complete.


 
 
 Let $\hat F \subset \mathbb{C}$ be a fundamental domain for the action by deck-translations $w\mapsto w+1$. 
 Then there are exactly $(n-2)$ horizontal rays of the $q_{sym}$-metric on $\mathbb{C}^\ast$  that are asymptotic to the  pole at the origin; these pull back to horizontal rays $H_1,H_2,\ldots H_{n-2}$ contained in $\hat F$.
 By Proposition \ref{rays}, the images of these under the harmonic map $\hat h$ determine  $(n-2)$ distinct ideal points  that we denote $\theta^0_1,\theta^0_2,\ldots \theta^0_{n-2}$.
Their images under the deck-translations shall be denoted by $\theta^j_1,\theta^j_2,\ldots \theta^j_{n-2}$ for each $j\in \mathbb{Z}$. 


 We start with the observation:
 
 \begin{lem}\label{thetas} The sequence of points $\{\theta^j_1,\theta^j_2,\ldots \theta^j_{n-2}\}_{j\in \mathbb{Z}}$ defined above are monotonic on the ideal boundary $\partial \mathbb{D}$, that is, lie in a cyclic order. In particular, we can define  $$\theta_+ = \lim\limits_{j\to \infty} \theta^j_{i} \text{ and } \theta_- = \lim\limits_{j\to - \infty} \theta^j_{i}$$ which are independent of the choice of $1\leq i\leq n-2$. These also the limit points of the image of the real axis $\mathbb{R}\subset \C$ under $\hat h$, respectively.
Moreover,  they are distinct ideal boundary points, that is, $\theta_-\neq \theta_+$. 
 \end{lem}
 
 \begin{proof}
 
The fact that the sequence of limit points of the boundary cusps are distinct points that are monotonic on the circle, and consequently define limiting points $\theta_\pm$, is a consequence of the fact that $\hat{h}$ is a diffeomorphism to its image (see Theorem \ref{tamwan}): for details, see the proof of Lemma 3.1 of \cite{Au-Tam-Wan}.

 For the second statement, consider the polygonal exhaustion at the pole of $q_{sym}$ at the origin (see Definition \ref{polyex}) such that for each $j\geq 1$, the boundary $\partial P_j$ comprises alternate horizontal and vertical segments in the $q_{sym}$- metric of length $L_j +O(1)$, where $L_j \to \infty$.
 
 Pulling back this exhaustion to the universal cover $\mathbb{H}^2$ gives a collection polygonal bi-infinite lines $\{\beta_j := \widetilde{\partial P_j}\}_{j\geq 1}$ that are each invariant under the translation $w\mapsto w+1$. See Figure 5. 
 
 Fix any of the lines $\beta_j$. Note that the vertical distance in the $\tilde{q}_{sym}$-metric of  $\beta_j$ from the real axis is uniformly bounded. Then, by the same proof as Proposition \ref{rays}, the limit points of the image of either end of the line $\beta_j$ are the same as the limit points of the positive and negative real axes. 
 
 As $j\to \infty$  the horizontal segments of $\beta_j$ limit to the horizontal lines of the $\tilde{q}_{sym}$-metric asymptotic to distinct directions in $\C$. Hence by Proposition \ref{prop-est} the images of these segments limit to the geodesic lines between the $\theta^j_i$s. In particular, the images of the endpoints of the horizontal segments of $\beta_j$ are close to $\theta^j _i$ in the Euclidean metric, say bounded above by $\ep>0$ (where $\ep\to 0$ as $j\to \infty$). Since the $\theta^j _i$   limit to $\theta_\pm$, the images of the endpoints of these horizontal segments limit to points $\ep$-close to $\theta_\pm$. Hence the  limit points of the images of $\beta_j$ are $\ep$-close to $\theta_-$ and $\theta_+$.  Since $\ep>0$ was arbitrary, and we have already observed that the limit points of the images of $\beta_j$ coincide with that of the real axis $\mathbb{R} \subset \mathbb{C}$, we conclude that those limit points are precisely $\theta_\pm$. 

\begin{figure}
  \centering
  \includegraphics[scale=0.45]{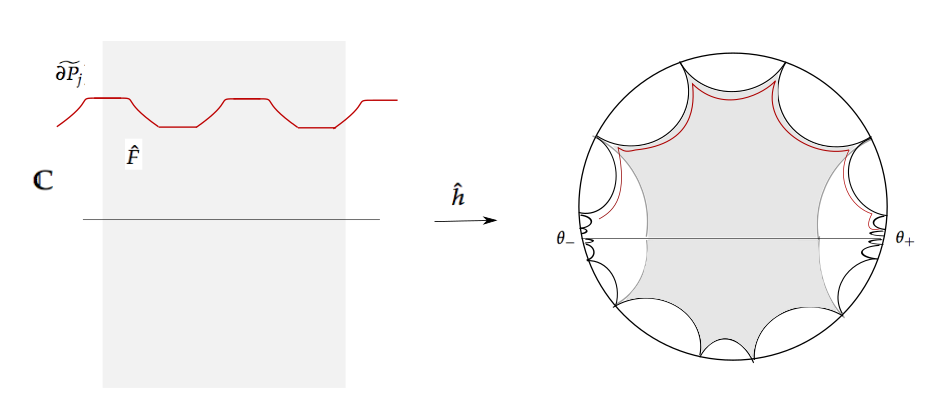}\\
  \caption{ The map $\hat h$ when restricted to the upper half plane $\mathbb{H}^2 \subset \mathbb{C}$ yields the desired equivariant harmonic map $\tilde{h}$ to a polygonal end. }
\end{figure}

Finally, to show that the limit points on either end are distinct,  recall that $\hat{h}$ is defined on $\mathbb{C}$, and hence there is a corresponding sequence of ideal points determined by the image of the restriction of $\hat{h}$ to the \textit{lower} half-plane. (By the symmetry of $q_{sym}$ the Hopf differentials on the upper and lower half planes are invariant under the conformal involution $w\mapsto -w$.) 

Assume that $\theta_-= \theta_+$. (In what follows we shall denote this point by $\theta_\pm$.)

Since $\hat{h}$ is a diffeomorphism, the (closure of the) image of the real line $\mathbb{R} \subset \mathbb{C}$ under $\hat{h}$ is a closed loop starting and ending at the point $\theta_\pm$. By the Jordan Curve theorem, this separates the Poincar\'{e} disk $(\mathbb{D},\rho)$ into two components, such that one of the components has $\theta_\pm$ as the only ideal boundary point in its closure.

Consider a bi-infinite path ${l}$  in $\hat{F}$ intersecting the real axis exactly once and asymptotic to a pair of horizontal rays in the $\hat{q}$-metric, one in the upper half-plane and one in the lower half-plane. Since $\hat{h}$ is a diffeomorphism, the image of the line $\hat{h}({l})$ is an embedded arc between the two ideal limit points determined by the image of the chosen horizontal rays on either side. Moreover, the arc intersects the image of the real line once. Then, by the above assumption, we have that one of the ideal limit points must be $\theta_\pm$. This is true for any choice of a pair of horizontal rays  that ${l}$ is asymptotic to, which contradicts the fact that the limit points for the images of the horizontal rays in the lower half-planes, are distinct (\textit{cf.} Proposition \ref{rays}).
 \end{proof}

 As a consequence, we obtain:
 
 \begin{lem}\label{Tinv}
 Suppose we have normalized $\tilde{h}:\C \to (\D, \rho)$ in \eqref{mapc} such that $\theta_\pm = \pm 1\in \partial \D$ and the  image of the positive vertical (imaginary) axis limits to $i$, the ideal point of cusp ``$1$". 
  Then there are hyperbolic isometries $T, S$ fixing $\theta_\pm$ such that 
 \begin{equation}\label{eq-asymp1}
 \tilde{h}(w+1) = T\circ \tilde{h}(w)
 \end{equation}
 and
 \begin{equation}\label{eq-asymp2}
 \tilde{h}(- w) = - S\circ \tilde{h}(w)
 \end{equation}
  for all $w\in \mathbb{C}$.

 \end{lem}
 \begin{proof}


By Theorem \ref{tamwan}  any two harmonic maps with identical Hopf differential must differ by a postcomposition with an isometry.  The Hopf differential $\tilde{q}_{sym}$ is invariant under the translation $w \mapsto w+1$ and involution $w\mapsto -w$. Hence  \eqref{eq-asymp1}  and \eqref{eq-asymp2} hold for some isometries $T$ and $S$. We need to verify that they are hyperbolic isometries, by showing they have exactly two fixed points $\pm 1$.


 Since the pullback differential $\tilde{q}_{sym}$ to the universal cover of $\mathbb{D}^\ast$ is invariant under deck-translations, it is preserved under the symmetry $w\mapsto w+1$. 
 Since the real axis is preserved under this translation, and by the previous Proposition its image has distinct limit points $\theta_\pm$ on the ideal boundary, it follows by considering sequences $w^+_n, w^-_n \in \mathbb{R}$ such that $w^+_n \to 1$ and $w^-_n \to -1$, that (\ref{eq-asymp1}) holds only if $T$ is an isometry which fixes the two points $\pm 1$.

By construction, the images  under $\hat{h}$ of the translates of the ray $H_1$ by these translations,  limit to the points $\theta^j_1$ for $j\in \mathbb{Z}$. 
Thus, $T$ necessarily has to take $\theta^j_i$ to its successive limit point $\theta^{j+1}_i$ for each $j\in \mathbb{Z}$, and is therefore a hyperbolic isometry translating along an axis with endpoints at $\pm 1$. 

Similarly,  $\tilde{h}(-w^+_n) \to -1$  and hence $S(1) = 1$,  and $\tilde{h}(-w^-_n) \to 1$ which implies $S(-1)=-1$. Hence $S$ is also a hyperbolic isometry fixing the boundary points $\pm1$. 
\end{proof}

\begin{figure}
  \centering
  \includegraphics[scale=0.38]{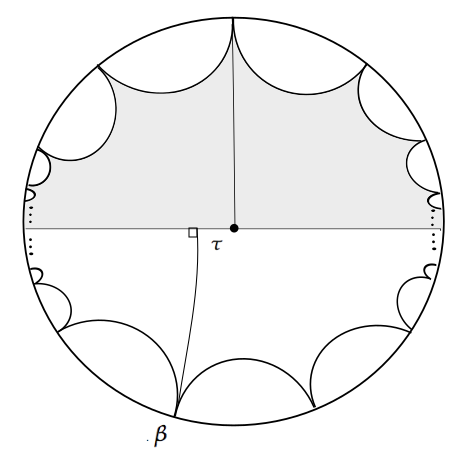}\\
  \caption{The images of the upper and lower half-planes differ by a hyperbolic translation that determines the boundary twist. }
\end{figure}

\textit{Remark.} The equivariant map $\hat{h}$ determines a shear parameter between the images of the upper and lower half-planes  as follows:  If the image under $\tilde{h}$ of the negative imaginary axis limits a point $\beta \in \partial \mathbb{D}$ in the lower semi-circle, then the shear parameter $\tau \in \mathbb{R}$ is the (signed) distance between $0$ and the foot of the perpendicular from $\beta$  to the geodesic between $\pm 1$.  (See Figure 6.) This parameter is also equal to the (signed) translation distance of the hyperbolic isometry $S$. This measures the ``boundary twist" parameter of the hyperbolic crown $\CC$ obtained in the quotient of the image of the upper half-plane by $T$. \\

Finally, we have:
  
  \begin{prop}[Defining $\Phi$]
  Given $q_{sym}\in  \mathrm{Q}_{sym}(P, n)$, there exists a harmonic map $$h:\mathbb{D}^\ast \to \CC$$ where $\CC$ is a hyperbolic crown, having Hopf differential $q$. 
  Moreover, the metric residue $a$ of $\CC$ is equal to  twice the real part of the residue of $q$.

  Together with the preceding remark assigning the boundary twist parameter, 
  we have a well-defined map
  \begin{equation*}
{\Phi}:  \mathrm{Q}_{sym}(P, n)   \to {\mathsf{Poly}_a}(n-2).
\end{equation*} 
where $P$ is the principal part of $q$.
  
  \end{prop}
  \begin{proof}
  In the preceding discussion, the equivariant map $\tilde{h}$ that is the restriction of $\hat{h}$ to $\mathbb{H}^2$  (see \ref{maph}) descends to the harmonic map $h$.
  
  The statement about metric residue  then follows from the final argument in the proof of Proposition \ref{poly-ex}: 
  
 Namely, consider a  polygonal exhaustion
  $$\cdots P_{j-1} \subset P_j \subset P_{j+1} \subset \cdots$$
   at the pole in $\D^\ast$, and its lift to the universal cover $\mathbb{H}^2$.
   
   By the distance estimates, we know that the lengths of the images by $\tilde{h}$  of a horizontal segment of length $l_h$  in the $\tilde{q}$-metric, is $2l_h + O(e^{-\alpha D})$, where $D$ is the distance of the segment from the zeroes of $q$. 
   
   Applying this to the horizontal sides of $\partial P_j$, note that  $D\to \infty $ as $j \to \infty$. Hence for any $\ep>0$, we can choose $j$ sufficiently large, such that the image of the horizontal sides of $\partial P_j$ have hyperbolic lengths  $2l_i + \ep$ for $1\leq i\leq (n-2)/2$.  Moreover, this image is $\ep$-close to a truncation of the crown $\CC$, and hence the metric residue $a$ of the crown (see Defn. \ref{metres}) which is independent of $\ep$ differs from the alternating sum of these lengths by $O(\ep)$.
   
 By Lemma \ref{qres}, this alternating sum of the lengths of the sides of $\partial P_j$ equals the real part of the quadratic residue at the pole, which is also independent of $\ep$.  Letting $\ep \to 0$ in the above equalities, we have that this real part equals $a/2$, as required. 
 
  Thus in the universal cover $\CC$ determines a polygonal end in ${\mathsf{Poly}_a}(n-2)$. 
 

\end{proof}

  
 
 
 

\subsection{Proof of Theorem \ref{asm} (Asymptotic Models)}

For $n$ even, the both spaces  $\mathrm{Q}_{sym}(P, n)$ and ${\mathsf{Poly}}_a(n-2)$ are homeomorphic to $\mathbb{R}^{n-3}$ and when $n$ is odd, both spaces are homeomorphic to $\mathbb{R}^{n-2}$. The map $\Phi$ is easily seen to be continuous (this follows, for example, by the argument for Lemma 2.2 in \cite{TamWan}). Hence to show that the map $\Phi$ between them is  homeomorphism, it is enough to show that the map is proper and injective.


\subsection*{Asymptotic geometry of the $q$-metric}

We begin with the following relation between the principal part and the singular flat geometry near the pole, namely, that relative distances in the induced metric are determined, up to a uniformly bounded error, by the principal part:

\begin{lem}\label{ageom1} Let  $n\geq 3$ and let $q_1,q_2 \in Q(P, n)$ be meromorphic quadratic differentials defined on $\D^\ast$ with the same principal part $P$. Then there exists $C_1>0$ such that  for any pair of points $z,w \in \mathbb{D}^\ast$ sufficiently close to the origin, we have
\begin{equation*}
\lvert d^h_1(z,w) - d^h_2(z,w) \rvert  \leq C_1
\end{equation*}
where $d^h_i$ denotes the horizontal distance with respect to the $q_i$-metric, for $i=1,2$. There is a similar  bound for differences of vertical distances $d^v_i$.

\end{lem}
\begin{proof}
Let $z,w\in B(0,r)$ where $0<r\ll 1$.
Recall that the vertical and horizontal distances between $z$ and  $w$ are the corresponding vertical and horizontal lengths of  a geodesic path $\gamma$ with respect to the $q_i$-metrics (see Definition \ref{h-v}):
\begin{equation*}
d^h_i(z,w) = \displaystyle\int_\gamma  \lvert \Re (\sqrt q_i) \rvert  \text{ and } d^v_i(z,w) = \displaystyle\int_\gamma  \lvert \Im (\sqrt q_i) \rvert 
\end{equation*}

The key is to observe that since the principal parts are the same, we have (\textit{cf.} \eqref{pp-eq}) :
\begin{multline}\label{hdiff}
\lvert d^h_1(z,w) - d^h_2(z,w)  \rvert   \leq  \displaystyle\int_\gamma  \lvert \Re (\sqrt q_1) - \Re (\sqrt q_2)  \rvert \\ \leq \displaystyle\int_\gamma \left\vert  \Re \left(\frac{\alpha_0}{\sqrt z}  +  {\alpha_{-1}}\sqrt z + O(z^{3/2}) \right) dz \right\vert  \leq C\displaystyle\int_\gamma \frac{\lvert dz\rvert}{\lvert z\rvert ^{1/2}} 
\end{multline}
for some constant $C>0$ which depends on the maximum modulus of the coefficients occurring in an expression of $q_1$ and $q_2$.
Since $\gamma \subset B(0,r)$, we have that the right hand side is $O(r^{1/2})$.

The same bound holds for the difference of the vertical distances.
\end{proof}

Conversely, we have that:

\begin{lem}\label{ageom2} Let  $n\geq 3$ and let $q_1,q_2 \in \mathcal{QD}(P, n)$ be meromorphic quadratic differentials defined on $\D^\ast$ having distinct principal parts $P_1, P_2$. Then there exists a sequence of points $z_i \to 0$ in $\D^\ast$ as $i\to \infty$ such that
\begin{equation*}
\lvert d^h_1(z_i,z_0) - d^h_2(z_i,z_0) \rvert  \to \infty
\end{equation*}
where $z_0$ is a fixed basepoint in $\D^\ast$. 
\end{lem}

\begin{proof}
Suppose $P_1$ and $P_2$ differ in the coefficient of $z^{-\nu}$ where $-\lfloor n/2 \rfloor \leq -\nu \leq -1$. 
Pick $z_0$, and the sequence $z_i$ to lie in a sector $\mathsf{S}$ in $\D^\ast$ where $\Re(\sqrt{q_1})$ and $\Re(\sqrt{q_1})$  are positive. 
Then 
\begin{multline}\label{hdiff2}
\lvert d^h_1(z_i,z_0) - d^h_2(z_i,z_0)  \rvert  =  \lvert \displaystyle\int_{\gamma_i}  \Re (\sqrt q_1) - \Re (\sqrt q_2)  \rvert\\ =  \lvert \displaystyle\int_{\gamma_i}  \Re (\sqrt q_1 - \sqrt q_2) \rvert \geq c\displaystyle\int_{\gamma_i} \frac{\lvert dz\rvert}{\lvert z\rvert ^{\nu}} 
\end{multline}
for a path $\gamma_i$ from $z_0$ to $z_i$ lying $\mathsf{S}$. 
The paths $\gamma_i$ tend to the origin as $i\to \infty$, and the integral on the right hand side diverges.
\end{proof}

\subsection*{Injectivity of $\Phi$}

The key proposition of this section is to observe that using the  preceding results concerning the asymptotic geometry near the pole, together with the estimates on the geometry of the harmonic map in \S2.5, we have:

\begin{prop}[Bounded distance]\label{inj} Fix $n\geq 3$ and a principal part $P \in \text{Princ}(n)$. Let  $\CC$ be a hyperbolic crown and $$h_1,h_2:\mathbb{D}^\ast \to \CC$$  be two harmonic maps with Hopf differentials $q_1,q_2 \in Q_{sym}(P,n)$. 

Then the  maps $h_1$ and $h_2$  are asymptotic in the following sense: 
\begin{equation}\label{dist-bd}
d_\rho(h_1(z), h_2(z)) \leq C_d
\end{equation}
for some constant $C_d>0$ where $d_\rho$ is the hyperbolic distance on $\CC$.



Conversely, if two such harmonic maps $h_1,h_2$ to the same crown are a bounded distance apart, then the principal parts of their Hopf differentials must be identical.
\end{prop}
\begin{proof}




Recall that by the construction in the preceding section,  the lifts to the universal cover $\tilde{h}_1$ and $\tilde{h}_2$ can be thought of as restrictions of harmonic maps $\hat{h}_1, \hat{h}_2: \mathbb{C} \to (\mathbb{D},\rho)$ to the upper half plane  $\mathbb{H}^2 \subset \C $.  

Moreover, their  Hopf differentials are the pullbacks of the differentials $(q_{sym})_{1}$ and $(q_{sym})_{2}$ respectively via the universal covering $\hat \pi:\mathbb{C} \to \mathbb{C}^\ast$.

The distance estimates of \S2.5 apply to the equivariant maps $\hat{h}_i$ for $i=1,2$, but  restricting to the upper half-plane and passing to the quotient the same estimates for \textit{upper} bounds of distances hold for the maps $h_i:\mathbb{D}^\ast\to \CC$. Henceforth we shall apply the estimates directly to the maps $h_i$ to obtain \eqref{dist-bd}. 

Fix a neighborhood $V \subset \mathbb{D}^\ast$ of the pole at zero such that any each point in $V$ 
is a distance $D\gg 0$ from any zero of ${q}_i$ (where $i=1,2$) in the induced quadratic differential metric.

Let $\{P_j\}_{j\geq 1}$ be the polygonal exhaustion for the ${q}_1$-metric on $\mathbb{D}^\ast$ as defined in  Definition \ref{polyex}.

By Lemma \ref{poly-ex} the image $h_1(P_j)$ is $\ep$-close to a truncation of the crown $\CC$, where $\ep\to 0$ as $j\to \infty$; these images form an exhaustion of the crown $\CC$.


It will suffice to show:\\

\textit{Claim. The maps $h_1$ and $h_2$ are bounded distance from each other on $\partial P_j$, where the bound is independent of $j$.}\\
\textit{Proof of claim.} 
Let $\alpha_j$ be the vertical side of $\partial P_j$ that maps into the cusp based at $1$ (for large $j$).

Fix a basepoint $z_0 \in V$. By construction, the horizontal distance of $\alpha_j$ in the $q_1$-metric from the basepoint is $L_j \pm O(1)$ for $j\geq 1$. 
Hence, the distance of the image $h_1(\alpha_j)$ from $h_1(z_0)$ is $2L_j +O(1)$ by Proposition \ref{prop-est}.

(Recall that in all this discussion $O(1)$ is a quantity that is independent of $j$.) 

Now by Lemma \ref{ageom1}, for any point in $\alpha_j$, the horizontal distance from $z_0$  with respect to the $q_2$-metric is also $L_j + O(1)$.
This implies that the distance of $h_2(\alpha_j)$ from $h_2(z_0)$ is approximately $2L_j + O(1)$  (again applying Prop \ref{prop-est}).

By our normalization both $q_1$ and $q_2$ have a leading coefficient $1$ (see (\ref{qform})), and hence the asymptotic directions of the horizontal rays at the pole are identical (\textit{cf.} the discussion in the example following Definition \ref{h-v}). 


Recall that  by Proposition \ref{rays},  the images of these horizontal directions tend towards the cusps of the polygonal end. 

Hence by the normalization we can assume without loss of generality that the horizontal ray corresponding that the midpoints of the sides $\alpha_j$ lie along, is mapped by $h_1$ and $h_2$  to the same cusp of $\CC$.

\begin{figure}
  \centering
  \includegraphics[scale=0.45]{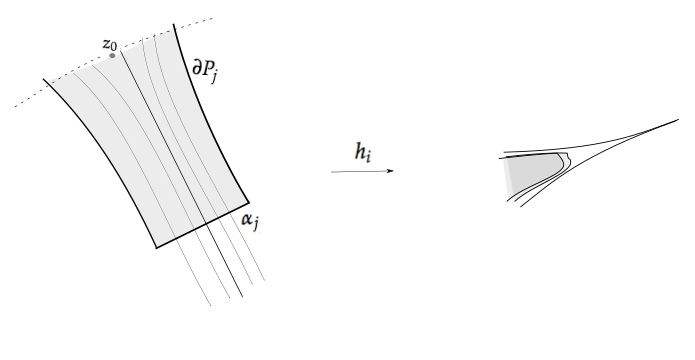}\\
  \caption{The maps $h_i$ for $i=1,2$ map the vertical side $\alpha_j$ of the polygon $\partial P_j$ into the same cusp; since the horizontal distances from $z_0$  for both images is  $L_j + O(1)$ with respect to the metric $q_i$, the distance between the images is uniformly bounded.}
\end{figure}

This observation implies that the two maps $h_1$ and $h_2$ are bounded distance apart when restricted to the side $\alpha_j$ of $P_j$. See Figure 7.


One can then extend the bounded-distance property to the remaining sides of $\partial P_j$:

Since the horizontal and vertical lengths of any side with respect to the two metrics are within bounded distance of each other,
by Prop \ref{prop-est} the distances between the images of the endpoints of the sides (and hence the images of the entire sides, since the horizontal sides map close to geodesic segments between these images, and the images of the vertical sides shrink) are bounded distance from each other. 
Consequently if one endpoint is already known to be a bounded distance apart, then so is the image of the entire side.
We can thus start by bounding the distance between the two maps on the sides adjacent to $\alpha_j$, and continue until one covers the entire polygon $\partial P_j$. $\qed$\\

By the claim, and the subharmonicity of the distance function between the maps $h_1$ and $h_2$, we can conclude that this distance function is uniformly bounded on $P_j$, where the bound is independent of $j$. Since the $P_j$s exhaust a neighborhood of the puncture in $\mathbb{D}^\ast$, we obtain a uniform bound over the entire domain, establishing \eqref{dist-bd}. 

The converse follows from Lemma \ref{ageom2} since if the principal parts are different, the diverging horizontal distances in the domain would imply, by  the geometric estimates (Proposition \ref{prop-est}), that the hyperbolic distances  between the corresponding image points also diverge, so the pair of maps are an unbounded distance apart in a neighborhood of the pole.\end{proof}


As a consequence, we obtain:

\begin{cor}[Injectivity of $\Phi$]\label{inj2} The Hopf differentials of the maps $h_1$ and $h_2$ in Proposition \ref{inj} are identical.
\end{cor}
\begin{proof}

By the previous proposition, the distance function between $\hat h_1$ and $\hat h_2$ is bounded; however such a distance function is subharmonic, and bounded subharmonic functions on $\mathbb{C}$ are constant, say equal to $c$.

If the constant $c\neq 0$,  then at every image point $h_1(z)$ one can place a unit vector towards $h_2(z)$. This defines a non-vanishing vector field $\mathcal{V}$ on the diffeomorphic image that is the doubled hyperbolic crown $\hat \CC$. Consider a horizontal horizontal side $\beta_j$  of $P_j$.   By Proposition \ref{prop-est} the images  of $\eta_j$ under both $h_1$ and $h_2$  tend towards a geodesic side $\gamma$  of the crown-end of $\CC$ as $j\to \infty$. In particular, if $z_j \in \beta_j$ is a sequence of points such that  $h_1(z_n) \to p \in \gamma$, then the images $h_2(z_n) \to p^\prime \in \gamma$ as well, where the distance between $p$ and $p^\prime$ is $c$.  Hence the vector field $\mathcal{V}$  is tangential to the geodesic sides of $\hat \CC$. 

Moreover, the vector field along the two geodesic sides bordering the same cusp point in the same direction, since by continuity this extends to a non-vanishing vector field in the intervening cusp,  with either all vectors towards the cusp or away from the cusp. Since the same geodesic line borders an adjacent cusp, this direction of the vector field  $\mathcal{V}$ (towards or away the ideal point of the cusp) alternates for alternate cusps. 

If  the number of cusps $n$ is odd, one cannot place signs that alternate for adjacent cusps, and hence this case is impossible.

If the number of cusps $n$ is even, by doubling the doubled hyperbolic crown by identifying two copies along the geodesic boundaries by an isometry, one obtains a  vector field on a sphere with total index $0$  (it is index $+1$ in the cusps where it points towards the puncture, and and $-1$ in the others, and these alternate) which contradicts the Poincar\'{e}-Hopf theorem.

 We thus obtain that  $c=0$, and the maps are in fact identical, and so are their Hopf differentials, and their restrictions to $\mathbb{H}^2 \subset \mathbb{C}$. Passing to the quotient using the equivariance, we obtain $q_1= q_2$ as desired.
\end{proof}

\medskip

\subsection*{Properness}











To prove properness of the map $\Phi$ in \eqref{phimap}, we need to show:

\begin{prop}\label{prop-er} Let $P \in \text{Princ}(n)$. Suppose there is a sequence of symmetric quadratic  differentials  $q_k$ (for $k\geq 1$)  with the same principal part $P$ at the pole at zero in $\D^\ast$, such that the images of the corresponding sequence of harmonic diffeomorphisms $h_k:\mathbb{D}^\ast \to \CC_k$  determine a sequence of hyperbolic crowns  $\CC_k$ that converge to a hyperbolic crown $\CC$ as $k\to \infty$.

Then after passing to a subsequence, the Hopf differentials $q_k$ converge in $Q_{sym}(P,n)$.
\end{prop}


\begin{proof}

Consider the associated sequence of $\mathbb{Z}$-equivariant harmonic maps
$$\tilde{h}_k:\mathbb{C} \to (\mathbb{D},\rho)$$
such that the restriction of each to $\mathbb{H}^2$  is an equivariant harmonic map to an image polygonal end.

Recall that since we are taking symmetric differentials, by our preceding construction,  the real axis in  $\mathbb{C}$ maps to the lift of the geodesic boundary of the corresponding crown. 

In what follows, for $0<r \leq 1$, let $\tilde{C}_r$ be the bi-infinite line in $\mathbb{H}^2$ that is the lift of the circle $C_r$ of radius $r$ in $\D^\ast$.  Let $F_r \subset \tilde{C}_r$ be a segment that is a fundamental domain of the action by translation $w\mapsto w+1$.


Our goal is to show that the sequence $\{h_k\}_{k\geq 1}$ is pointwise uniformly bounded. 

We shall show that any pair of maps $h_k,h_l$ in the sequence are a bounded distance apart, where the bound is independent of which pair we chose.  

(Note that the fact that such a pair is a bounded distance apart is already  consequence of Lemma \ref{inj} since the corresponding Hopf differentials $q_k$, $q_k$ have the same principal part $P$. The uniform bound shall crucially depend on the additional assumption that the image polygonal ends converge.) \\


\textit{Claim 1. There is a constant $C>0$ (independent of the pair) and a sufficiently small $0<r\ll 1$, such that the diameter of the images of $F_r$  the maps  $\tilde{h}_k, \tilde{h}_l$ is bounded above by $C$.}\\
\textit{Proof.} 
As a consequence of the distance estimates in  Lemma \ref{ageom1}, we have that for sufficiently small $r>0$, the 
 maximum vertical and horizontal distance between points on $F_r$ with respect to both the $q_k$- and $q_l$-metrics, is bounded above by some $C_0>0$.  
 
The estimates of the geometry of harmonic maps in Proposition \ref{prop-est} hold if $r$ is sufficiently small, since any point in $F_r$ is  then far (say a distance  $D\gg 0$) from the zeroes of the differential (which are determined by $P$). Thus, we obtain the distance bound
\begin{equation}\label{D2-2}
d_\rho(\tilde{h}_k(x), \tilde{h}_k(y)) < C_0 e^{-\alpha D} + 2C_0 + O(e^{-\alpha D})  =: C
\end{equation}
for any pair of points $x,y\in F_r$.  (Note that the bound is  independent of the pair.)
$\qed$\\

Second, a crucial observation is:\\

\textit{Claim 2. For any $r>0$, the images of $F_r$ under the maps $\tilde{h}_k, \tilde{h}_l$ intersect a fixed compact set in $(\D,\rho)$ .}\\
\textit{Proof.} 
First, since the image polygonal ends converge to a polygonal end $\P$,  the images of  $F_1$  converge to a geodesic segment on the boundary of $\P$ corresponding to a fundamental domain of the $\mathbb{Z}$-action on $\P$. In particular, the images of $F_1$ remain uniformly bounded in  $(\D,\rho)$.

Pick a basepoint $z_0$ in $(\D,\rho)$ (say, lying on the limiting image of $F_1$).
There is a constant $D_2>0$ such that there is a geodesic side of $\P$ at a distance at most $D_2$ from $z_0$, and by the convergence of the images to $\P$,  the same is true for the corresponding geodesic side  $\gamma_k$ of $\P_k$ for all sufficiently large $k$.

The image of the circle $C_r$ for any $r>0$  under ${h}_k$ is homotopically non-trivial curve in $\CC$, and in particular, must exit the cusps bordering $\gamma_k$.  (See Figure 8.) 
Hence in the universal cover the image of the fundamental segment $F_r$ that intersects the ball of radius $D_2$ around $z_0$, for each $k$.  $\qed$ \\

\begin{figure}
  \centering
  \includegraphics[scale=0.45]{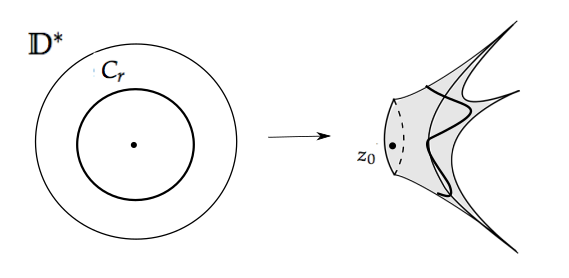}\\
  \caption{The image of the circle $C_r$ under the harmonic diffeomorphism $h_k$  must exit  the cusps and has a point $x_k$ at a distance  $D_2$ from the basepoint $z_0$. Moreover the diameter of the image is uniformly bounded since the principal part of $\text{Hopf}(h_k)$ is independent of $k$. }
\end{figure}

As a consequence of Claims 1 and 2, we obtain:\\

\textit{Claim 3. There is a constant $C_1>0$ (independent of the pair) and a sufficiently small $0<r\ll 1$, such that on $F_r$, the restriction of the maps  $\tilde{h}_k, \tilde{h}_l$ have a distance bounded above by $C_1$.}\\
\textit{Proof.} 
As a consequence of Claim 2, for any $r>0$, there exists a point $x_0 \in F_r$ such that 
\begin{equation}\label{D2}
d_\rho(\tilde{h}_k(x_0), z_0) < D_2 \text{ and } d_\rho(\tilde{h}_l(x_0), z_0) < D_2
\end{equation}
where $z_0$ is a fixed basepoint in $(\D,\rho)$.
This implies that 
\begin{equation}\label{D2-1}
d_\rho(\tilde{h}_k(x_0), \tilde{h}_l(x_0) ) < 2D_2. 
\end{equation}
By Claim 1 (see \eqref{D2-2}) we then that for $r$ sufficiently small, we have 
\begin{align}
d_\rho(\tilde{h}_k(x), \tilde{h}_l(x)) \leq  d_\rho(\tilde{h}_k(x), \tilde{h}_k(x_0)) + d_\rho(\tilde{h}_k(x_0), \tilde{h}_l(x_0) ) + d_\rho(\tilde{h}_l(x_0), \tilde{h}_l(x)) \\ < 2C + 2D_2 =: C_1
\end{align}
for any $x\in F_r$,  where the right hand side is uniform over the sequence (that is, independent of the pair $h_k$, $h_l$). $\qed$\\

We have already noted that the distance between the maps $\tilde{h}_k, \tilde{h}_l$  has a uniform bound on $F_1$ (where ``uniform" means that the bound is independent of the pair chosen from the original sequence). Claim 3 implies that there is a similar uniform bound on this distance function when the maps are restricted to $F_r$ (for a sufficiently small $r$).

Passing to the quotient, we obtain a function on the annular region bounded by $\partial \mathbb{D}$ and $C_r$ that has the above uniform bounds on the two boundary components. However the  distance function between two harmonic maps is subharmonic, and the Maximum Principle then implies that there is a uniform bound on the entire annular region. 

We can thus conclude that  the original sequence ${h}_k$ is pointwise bounded. By standard interior gradient estimates (see \cite{Cheng}, or Theorem 2.5 of \cite{AHuang}) all derivatives of the harmonic maps are also uniformly bounded on compacta. Thus by Arzela-Ascoli there is a convergent subsequence, as claimed, and we obtain a limiting harmonic map $h$. 
\end{proof}

\subsection{Revisiting polynomial quadratic differentials}
Having established Theorem \ref{asm}, we note that an application of some of the methods introduced to clarify the fibers of the map 
\begin{center}
$\{$ polynomial quadratic differentials $\} \longrightarrow \{$  ideal polygons $\}$
\end{center}

that is implied by Theorem \ref{image}. \\

To be more specific, let $n\geq 6$ be the order of the pole at infinity.  If $\text{PDiff}(P,n)$ is the space of polynomial quadratic differentials of degree $(n-4)$  with fixed principal part $P$ at the pole at infinity. From the argument in Lemma \ref{dims}, such a polynomial is determined by $(n-r-3)$ complex coefficients, where $r = \lfloor n/2\rfloor$, namely $(a_0, a_1,\ldots a_{n-r-4})$. That is, it is determined by $(n-6)$ real parameters when $n$ is even, and $(n-5)$ real parameters when $n$ is odd. 

(Note that the total space of monic, centered polynomials as in \eqref{polydiff} is homeomorphic to $\mathbb{C}^{d-1} \cong \mathbb{R}^{2d-2} \cong \mathbb{R}^{2n-10}$ since $d=n-4$.) 

On the other hand the space $\text{PGon}(n-2)$ of ideal $(n-2)$-gons is determined by $(n-5)$ parameters, since $3$ of the points can be fixed after a normalization, and is hence homeomorphic to $\mathbb{R}^{n-5}$.  When $n$ is even, those with prescribed metric residue is  is hence homeomorphic to $\mathbb{R}^{n-6}$ (\textit{cf.} the corresponding discussion for hyperbolic crowns in Lemma \ref{poly1}).

We already know from Theorem \ref{image} (proved in \cite{HTTW}) that there is  then a well-defined map 
\begin{equation}\label{Pieq}
\Pi: \text{PDiff}(P,n) \to \text{PGon}(n-2)
\end{equation}

that assigns to a polynomial quadratic differential $q\in \text{PDiff}(P,n)$ the image of the harmonic map $h:\C\to (\D,\rho)$ with Hopf differential $q$.  \\

We can now assert the following:

\begin{prop}\label{prop-pdiff} Let $n\geq 6$ and fix an ideal polygon $\PP\in  \text{PGon}(n-2)$. Then 
for any choice of a principal part $P \in \text{Princ}(n)$  that is compatible with the ideal polygon  (see Definitions \ref{qd-def} and \ref{compat}) there exists a unique harmonic diffeomorphism  $$h:\C \to \PP$$
whose Hopf differential has principal part $P$. 

Equivalently, the map $\Pi$ in \eqref{Pieq} is a homeomorphism. 

\end{prop}

\begin{proof}[Sketch of the proof]

Recall  that Proposition \ref{inj}  shows that two harmonic maps on a punctured surface with a Hopf differential having a pole at the puncture with the same principal part $P$ are a bounded distance apart: its proof  relies on the fact that distances between image points are determined by the real and imaginary parts of integrals of $P$, and works \textit{mutatis mutandis} for maps from $\C$ to $(\D,\rho)$.

The uniqueness is then a consequence of the fact that the distance function between two harmonic maps is subharmonic, and a bounded subharmonic function on $\C$ is constant.  Moreover, the fact that the constant is zero can be derived from the same argument as in Corollary \ref{inj2}: namely, if the constant is non-zero, the vector field on the image polygon constructed by considering the unit vector in the direction of the difference of the two maps at each point is non-vanishing, and points into and out of each alternate cusp. This contradicts the Poincar\'{e}-Hopf theorem  applied to the punctured sphere obtained by  doubling the polygon.

Moreover, the arguments of \S3.2 can be applied to show the existence part:


Namely, from the discussion preceding this Proposition, the dimensions of the spaces on the either side of \eqref{Pieq} match. Moreover $\Pi$  is continuous and injective by the preceding discussion, and properness  is a consequence of the argument in Proposition \ref{inj2} (see pg. 27 ) applied to a sequence of harmonic maps from $h_k:\C\to (\D, \rho)$ with the image polygons converging to an ideal polygon, after making the appropriate changes (for example, the image of the curve $C_r$  is not non-trivial in homotopy, but has winding number $1$ around the origin, and in particular must exit the cusps).
Then by the Invariance of Domain, $\Pi$ is homeomorphism and in particular, surjective. 
\end{proof}

\textit{Remark.} It is worth observing that the above argument implies that the $\rho$-distance between two \textit{distinct} harmonic maps from $\C$ to $(\D,\rho)$ with the same polygonal image  is unbounded; in particular, the maps escape out the cusps at different rates.







\section{Proof of Theorem \ref{thm1}}





The goal of this section is to prove Theorem \ref{thm1}.




For convenience of notation, we shall assume that  the number of marked points $k$ equals $1$, that is we shall prove:

\begin{prop}\label{main-prop}  Let $X$ be a marked Riemann surface of genus $g\geq 1$ and one marked point $p$, with a fixed coordinate disk   $(U,p) \cong (\mathbb{D},0)$. Let $Y$ be a marked crowned hyperbolic surface with one crown end $\CC$ having $(n-2)$ boundary cusps where $n\geq 3$. Let $P \in \text{Princ}(n)$ be a principal part. 
Then there exists
a unique harmonic diffeomorphism $$h:X\setminus p \to Y$$  taking a neighborhood of  $p$ to the crown end $\CC$ and preserving the marking, such that its Hopf differential has a pole of order $n$ at $p$ with principal part $P$. 
\end{prop}

\medskip

By Theorem \ref{asm} (Asymptotic Models), for any  principal part $P \in \text{Princ}(n)$, there exists a harmonic ``model map" 
$$m:\mathbb{D}^\ast \to \CC$$ 
that is asymptotic to its crown end, and with Hopf differential $q$ having principal part $P$.

 Our method of proving Proposition \ref{main-prop} involves taking an exhaustion of the punctured surface, solving a Dirichlet problem for each compact surface with boundary, where the boundary condition is determined by the model map $m$. The main work lies in showing that the sequence of harmonic maps has a convergent subsequence, which in turn relies on proving a uniform energy bound on any compact subsurface of $X\setminus p$.

 \subsection{Defining the sequence $h_i$}

We begin by choosing an exhaustion of the  disk $U$ with nested  sub-disks $\{U_i\}_{i\geq 0}$ of decreasing radii, where $U_0 :=U$. Defining $X_i = X\setminus U_i$, we obtain a compact exhaustion of the surface $X\setminus p$, namely, a nested collection 
\begin{equation}\label{xi-exh}
X_0 \subset X_1 \subset \cdots X_i \subset \cdots 
\end{equation}
of compact subsurfaces with boundary, such that $\bigcup\limits_{i\geq 0} X_i = X \setminus p$.\\

We shall denote the annulus $A_i := X_i \setminus X_0$.\\




Define 
 ${h_i}:{X_i} \to Y_i$  be the harmonic map  that preserves the marking, with the boundary condition that it restricts to ${m}$ on $\partial X_i$. \\

\textit{Note.} Such a harmonic map exists by the work of Lemaire in \cite{Lem82}. Moreover, since the target is negatively curved, the usual convexity of energy along a geodesic homotopy implies that in fact it is the \textit{least energy} map with the given boundary conditions. We shall use this property of $h_i$ in \S4.3. \\

\medskip

The crucial step  for the uniform energy bound is proving an comparison of the model map $m$ with the solution of a ``partially free boundary problem"  on a cylinder, that we describe in the next section. In previous work  (\cite{GW1}, \cite{GW2}) we had also considered partially free boundary problem to certain metric trees. However, the argument there used the fact that the maps had an additional symmetry; which ensured a decay of the maps along the cylinder. Here, in contrast, we exploit our knowledge of the model map $m$ obtained from our analysis in \S3.

\subsection{Partially free boundary problem}



Recall that  $\mathcal{C}$ be a hyperbolic crown with $(n-2)$ cusps and a chosen basepoint $q$.\\

We introduce the following (\textit{cf.} \S3.3 of \cite{AHuang}):

\begin{defn}[Partially free boundary problem]\label{pfd}  
Let $A$ be an annulus of modulus $M>0$, and let $f:\partial_+ A \to \mathcal{C}$ be a $C^1$-map. 

Let $$\phi:A \to \CC$$ be the least energy map with the boundary condition  $\phi\vert_{\partial_+ A} = f$ and  no restriction on $\partial_-A$ ( the ``free boundary").
Then $\phi$ will be referred to as the solution to the ``partially free Dirichlet boundary problem" (PFD for short).
\end{defn}

Such a map is harmonic by virtue of having least energy. Moreover,  A. Huang in \cite{AHuang} makes the crucial observation that by a variational argument one can show that $\nabla \phi \vert_{\partial A_-} ( \nu) = 0$ where $\nu$ is the outward-pointing normal vector along the free boundary. This allows one to frame the following equivalent Dirichlet problem:

\begin{lem}[Doubling lemma, Lemma 3.5 of \cite{AHuang}]\label{doubl}
Consider the doubled annulus $\hat{A}$  obtained by two copies of $A$ identified along the free boundary $\partial_-A$ in each by a reflection. 
Let $$\Phi:\hat{A} \to \mathcal{C}$$ be the least energy map with the  boundary condition $f$ on both boundary conditions.

Then the solution $\phi$ of the partially free boundary problem is the restriction of $\Phi$ to the annulus $A \subset \hat{A}$. 
\end{lem}

For the key estimate to follow, we shall need the following easy fact (we shall state  it in slightly greater generality than we need):

\begin{lem}\label{htw}
Suppose $A =\{(x,\theta) \vert 0\leq x\leq L, \theta \in S^1\}$ is a Euclidean cylinder of circumference $2\pi$ and length $L\geq 1$.  Let $$f:A\to Y$$ be a smooth map to a target space that is equipped with a Riemannian metric, and for $\tau>0$ let $$\Theta_\tau: A\to A$$ be the ``twist" map  defined by $\Theta_\tau(x, \theta) =  (x, \theta + \tau x/L)$. 

Then we have the energy estimate
\begin{equation}\label{en-comp-1}
\mathcal{E}(f) - K_1 \leq \mathcal{E}(f \circ \Theta_\tau) \leq  \mathcal{E}(f) + K_1
\end{equation}
where $K_1>0$ is a constant that is independent of $L$ (but depends on $\tau$). 
\end{lem}
\begin{proof}
Since $$d\Theta_\tau = \begin{pmatrix} 1 & \tau/L \\ 0 & 1 \end{pmatrix}$$
we can easily compute that the norm of the derivatives satisfy 
\begin{equation*}
\lvert \lVert d(f\circ \Theta_\tau) \rVert - \lVert df \rVert \rvert =  O(1/L)
\end{equation*}
at each point of $A$. 
In particular, 
\begin{multline}
\mathcal{E}(f \circ \Theta_\tau)  = \displaystyle\int\limits_{A} \lVert d(f\circ \Theta_\tau) \rVert^2 \leq  \left( \displaystyle\int\limits_{A} \lVert df \rVert^2 \right)    + O(1/L)\cdot 2\pi L  = \mathcal{E}(f) + O(1) 
\end{multline}
which is half of \eqref{en-comp-1}, and the proof of the other inequality is similar. 
\end{proof}

In words, the Lemma above shows that for a map defined on a cylinder $\mathsf{C}$ , precomposing the map with a fixed twist of $\mathsf{C}$ increases the energy by a bounded amount that is independent of the modulus of $\mathsf{C}$.\\

In what follows, we shall need the following equivariant version of the previous lemma obtained by passing to the universal cover:

\begin{cor}\label{htw2}
Let $S_L = \{ \lvert \Im(z) \rvert \leq L \}$ be an infinite strip, and let $$\tilde{f}: S_L \to (\mathbb{D},\rho)$$ be a smooth map  that is equivariant with respect to the $\mathbb{Z}$-action by the translation $z\mapsto z+1$ on the domain and by a hyperbolic translation $T$ in the target, where $T$ is a hyperbolic isometry with axis the geodesic line from $-1$ to $1$. 

For $\tau>1$, let $\tilde{\Theta}_\tau:S_L\to S_L$ be the affine map that restricts to the identity map on the boundary component $\{\Im(z) =L\}$ and is a translation by $\tau$ on the boundary component $\{\Im(z) = -L\}$.

Then the equivariant energy satisfies:
\begin{equation}\label{en-comp}
\hat{\mathcal{E}}(\tilde{f}) - K_1 \leq \hat{\mathcal{E}}(\tilde{f} \circ \tilde{\Theta}_\tau) \leq  \hat{\mathcal{E}}(\tilde{f}) + K_1
\end{equation}
where $K_1>0$ is a constant that is independent of $L$ (but depends on $\tau$). 
\end{cor}

Using this, we now prove:

\begin{prop}[Energy comparison]\label{comparison}
Let $A_i$ be the annular region constructed in the exhaustion \eqref{xi-exh}, with the two boundary components  $\partial_- A_i = \partial X_0$ and $\partial_+ A_i  = \partial X_i$. 
Let $\phi_i$ be the solution of the PFD on  $A_i$ with the boundary condition  on $\partial_+ A_i$ that is the restriction of the model map $m:\mathbb{D}^\ast \to \mathcal{C}$. 

Then we have
\begin{equation}\label{lbd}
\mathcal{E}(\phi_i) \leq \mathcal{E}(m\vert_{A_i}) \leq \mathcal{E}(\phi_i) + K_2
\end{equation}
where $K_2>0$ is uniformly bounded (independent of $i$).
\end{prop}

\begin{proof}
The lower bound is immediate from the property of $\phi_i$ being least energy.


For the other inequality, 
recall from \S3.1 that there is an equivariant harmonic model map $\hat{m}$ on $\mathbb{C}$ with Hopf differential $\tilde{q}_{sym}$, such that 

\begin{enumerate}

\item $m$ is the quotient of the restriction of  the translation invariant map $\hat{m}$ defined on $\C$, to the upper half-plane $\mathbb{H}^2$, and

\item The values of $\hat{m}(-z)$ differs from $-\hat{m}(z)$ by a hyperbolic translation $S$.

\end{enumerate}



Restrict the harmonic map $\hat{m}$ to the strip $S_L = \{ \lvert \Im(z) \rvert \leq L \}$, where the quotient of $S_L \cap \mathbb{H}^2$ by the translations yields the annulus $A_i$
By the translation invariance, this restriction descends to a  quotient map ${m}^d_i$ that is harmonic on the doubled annulus $\hat{A}_i $ (the quotient of $S_L$ by the translations), and is the least energy map having those boundary conditions. 


On the other hand, by the Doubling Lemma, $\phi_i$ is also the restriction to $A_i$ of a map ${\Phi}_i$ on the doubled annulus  $\hat{A}_i $, which has identical boundary conditions on the two boundary components (both equal to $m\vert_{\partial_+A_i})$.

The idea is to then modify ${\Phi_i}$ by a suitable twist $\Theta_\tau$, such that the lift to $S_L$ has the same boundary conditions as $\hat{m}$ on $S_L$. Note that the amount of this twist is uniformly bounded (independent of $i$) as it only depends on the boundary twist data of the crown $\CC$. 

By the previous lemma, this composition ${\Phi_i} \circ \Theta_\tau$ has  energy only a uniformly bounded amount (say $K_1>0$) more than that of   ${\Phi}_i$. Since  the composition has exactly the same boundary conditions as $m^d_i$ on $\hat{A}_i $, we have from the energy-minimizing property of $m^d_i$ noted above, that
\begin{equation}
 \mathcal{E}(m^d_i) \leq \mathcal{E}(\Phi_i) + K_1
\end{equation}

from which the right-hand side of  \eqref{lbd} follows with $K_2 = \frac{K_1}{2}$, as the doubled maps have energy exactly twice as those defined on the annulus $A_i$. 
\end{proof}

\subsection{Proof of Proposition \ref{main-prop}}


\subsection*{Convergence}
 
To show that $h_i$ uniformly converge to a harmonic map $h$ after passing to a subsequence,
we need \textit{a priori} energy bounds of the harmonic maps.

To do this, we shall use the energy comparison proved in Proposition \ref{comparison}. The rest of the argument is originally due to Wolf in \cite{Wolf3}, and used by Jost-Zuo in \cite{Jost-Zuo}, and in our previous work  \cite{GW1}, \cite{GW15}, \cite{GW2}.





\begin{prop}[Uniform energy bounds]\label{ebound} For the sequence of maps $h_i$ constructed in \S4.1, there is a uniform energy bound on compacta, that is, for any compact subsurface $K \subset X \setminus p$, we have
\begin{center}
$\mathcal{E}(h_i \vert_K) \leq C$ 
\end{center}
for each $i\geq 1$, where $C$ is independent of $i$.

\end{prop}

\begin{proof}

As observed in \S4.1, $h_i$ is the least energy map with the prescribed boundary condition on $\partial X_i$, namely, that the boundary map agrees with $m\vert_{\partial X_i}$.

For any $i\geq 1$ one can define a map $g$ that is a candidate solution to this energy-minimizing problem by defining $g$ to equal $h_1$ on $X_1$, and restrict to the map $m$ on $A_i = X_i \setminus X_1$. (Note that this defines a continuous map since $h_1\vert_{\partial X_1} = m\vert_{\partial X_1}$; since $\partial X_1$ is an analytic curve of measure zero, the candidate map $g$ is locally square-integrable.

We then have
\begin{equation*}
 \mathcal{E}(h_i) \leq  \mathcal{E}(g)  =  \mathcal{E}(h_1)  + \mathcal{E}(m\vert_{A_i}).
\end{equation*}

On the other hand, we have the equality
\begin{equation*}
\mathcal{E}(h_i) = \mathcal{E}(h_i\vert_{X_1})  + \mathcal{E}(h_i\vert_{A_i})  
\end{equation*}

and by the energy-minimizing property of the solution to the partially free boundary problem, 
\begin{equation*}
\mathcal{E}(\phi_i) \leq  \mathcal{E}(h\vert_{A_i})
\end{equation*}

Combining  these with energy comparison (\ref{lbd}) of Proposition \ref{comparison}  then yields 
\begin{equation*}
\mathcal{E}(h_i\vert_{X_1}) \leq K_2 + \mathcal{E}(h_1)
\end{equation*}
where the right-hand-side is independent of $i$.

Thus the energies restricted to the fixed subsurface $X_1$ is uniformly bounded. 
A similar argument then yields uniform energy bounds on \textit{any} compact subset $K$ of $X$, namely,  in the preceding argument, we replace the compact subsurface $X_1$ by a compact subsurface $X_m$ (for some $m>1$) in the exhaustion, that contains $K$. 
\end{proof}

We can now show:

\begin{lem} The sequence of harmonic maps $h_i$  (for $i\geq 1$) converges uniformly on compact sets, after passing to a subsequence, to a harmonic map $h:X\setminus p \to Y$.
\end{lem}
\begin{proof}
Let $K\subset X\setminus p$ be a compact set, and consider the restrictions $h_i\vert_K$. 

By Proposition \ref{ebound}, these restrictions have a uniform energy bound. By an application of the Courant-Lebesgue Lemma (see Lemma 3.7.1 of \cite{Jost0}) we can conclude that the sequence is equicontinuous. 

To show sub-convergence, it suffices to show that the images of a fixed basepoint $x_0 \in X\setminus p$ remain in a subset of uniformly bounded diameter in $Y$:

Consider a fixed compact subsurface $X_m$ (in the compact exhaustion constructed above)  containing $x_0$. We claim that the images of $X_m$ under $h_i$ (where $i\geq 1$) intersect a fixed compact subset $K_Y$ of $Y$:\\
This follows because of the non-trivial topology of the subsurface $X_m$, and the observation that  there exists a compact set $K_Y$  of $Y$ whose complement does not contain any curve that is homotopically non-trivial on the surface $Y$, except a curve that is homotopic to the boundary of the crown.
If the  image of $X_m$  under $h_{i}$ is disjoint from $K_Y$, then a non-trivial loop in $X_m$ that is not homotopic to $\partial X_m$ must map to either a curve that is null-homotopic or a curve that is homotopic into the puncture, which contradicts the fact that $h_{i}$ is a diffeomorphism preserving markings. 

Hence there is a sequence of points $x_i \in X_m$ such that $h_i(x_i) \in K_Y$. By the equicontinuity of the sequence, since $X_m$ is fixed,  there is a uniform bound on the distance between the images of the points $x_0$ and $x_i$ (for $i\geq 1$). In particular, the images of $x_0$ are uniformly bounded, as required.
\end{proof}

\subsection*{Finishing the proof}



 \begin{lem}  The restriction of $h_i$ to $U\setminus p \cong \mathbb{D}^\ast$ is a uniformly bounded distance from the model map $m:\mathbb{D}^\ast \to \CC$.   Moreover, the principal part of the Hopf differential of $h$ is $P$.

 \end{lem} 
 
 \begin{proof}
 Since $h_i \to h$ uniformly on compact sets, the sequence of maps $h_i$ restricted to  $\partial_- A_i = \partial X_0$ remains a bounded distance from $m$.
 By construction, any harmonic map $h_i$ in the sequence equals the restriction of $m$ on $\partial_+A_i$. Hence, by the Maximum Principle, we have that $h_i$ is uniformly  bounded distance from $m$ on the annulus $A_i$, where the uniformity means that the bound is independent of $j$. 
 We can also conclude that the limiting map $h$ is a bounded distance from $m$ on each $A_i$, and consequently on  the union $U \setminus p = \bigcup\limits_i A_i$.

By the statement of the ``converse"  in Proposition \ref{inj}, the Hopf differential of $h$ also has a principal part $P$. 
  \end{proof}

\medskip

\begin{lem}[Uniqueness] If there are two harmonic maps $h_1,h_2:X\setminus p \to Y$ such that their Hopf differentials $q_1, q_2$ have a principal part $P$, then the two maps are identical. In particular, $q_1=q_2$.
\end{lem}

\begin{proof}
A proof identical to that of Proposition \ref{inj} shows that the two maps are a bounded distance apart. Since a punctured Riemann surface is parabolic in the potential-theoretic sense, the distance function between them is bounded and subharmonic, and therefore constant.  If the constant is non-zero, then the same argument as in Corollary \ref{inj2} produces a non-vanishing vector field $\mathcal{V}$ on $Y$. This vector field is parallel along the geodesic boundary components of the crown end, and points alternately towards, and away from, the ideal cusps. On the cusped hyperbolic surface $\hat{Y}$ obtained by doubling, this produces a non-vanishing vector field  $\widehat{\mathcal{V}}$ that has index $\pm1$ on each puncture (that is a cusp in the hyperbolic metric) such that the sum of the indices lies in $\{-1,1,0,1\}$. Since the Euler characteristic of $\hat Y$ is $2-2\cdot(2g) -k$ where $g\geq 1$ and the number of cusps $k=1$, this contradicts the Poincar\'{e}-Hopf theorem.
\end{proof}

This completes the proof of  Proposition \ref{main-prop}; the argument extends \textit{mutatis mutandi} to the case of more than one marked point, or the case when $X$ has genus zero and at least three marked points,  proving Theorem \ref{thm1}.

\section{Proof of Theorem \ref{thm2}}

Fix a Riemann surface $X$ as in theorem, and as before, we shall assume that there is a single marked point $p$. 

Fix an $n\geq 3$ and a principal part $P \in \text{Princ}(n)$.
in what follows $D$ shall denote the divisor $-n\cdot p$.

We shall denote the relevant space of quadratic differentials by $\Q(X,D, P)$
: each element $q$ in this space has a pole of order $n\geq 3$ at $p$ with principal part $P$.\\

\begin{lem}\label{qxdim} The space $\Q(X,D, P)$ is homeomorphic to $\mathbb{R}^{6g-6+ n+1}$ when $n$ is odd, and $\mathbb{R}^{6g-6+ n}$ when $n$ is even. 
\end{lem}
\begin{proof}
Fix a coordinate disk $U$ around $p$.
The meromorphic quadratic differentials with a pole of order $1$ at $p$ is a vector space of complex dimension $3g-3+ 1$, and thus homeomorphic to $\mathbb{R}^{6g-6+2}$. By Lemma \ref{dims} the space of meromorphic quadratic differentials on $U$ of the form $\sum\limits_{i=2}^n a_i z^{-i}$ and prescribed principal part is homeomorphic to $\mathbb{R}^{n-1}$ parameters when $n$ is odd, and $\mathbb{R}^{n-2}$ parameters when $n$ is even. Moreover, any quadratic differential in $\Q(X,D, P)$  can be uniquely as such a sum of meromorphic quadratic differential on $X$ with a pole of order $1$ at $p$, and one on $U$ which lies in $Q(P,n)$; it is easy to see that any such pair can be obtained.
\end{proof}

By Theorem \ref{thm1}, for any crowned surface $Y\in \T(P)$, there exists a harmonic diffeomorphism $h:X\setminus p \to Y$ with Hopf differential $q\in \Q(X, D,P)$.  That is, $q$ has a pole of order $n\geq 3$ at $p$ with principal part $P$.

In this section we shall prove:

\begin{thm}[Theorem \ref{thm2} for one marked point]\label{thm2-case}
The resulting map
\begin{equation}\label{hatpsi}
\hat{\Psi}: \T(P) \to \Q(X,D,P)
\end{equation}
is a homeomorphism.
\end{thm}

\textit{Strategy of proof.} Observe that by Corollary \ref{wtdim} and Lemma \ref{qxdim} both sides are homeomorphic Euclidean space of the same dimension (which depends on $n$ even or odd). Therefore, by the Invariance of Domain,  it suffices to prove that the map is proper and injective.\\

\subsection*{Injectivity}
To show that $\hat{\Psi}$ is injective we use a standard argument using the Maximum Principle (see pg. 454 of \cite{Wolf0}).

We recall the argument:

Namely, splitting the energy density into the holomorphic and anti-holomorphic parts:
\begin{equation}\label{hl}
\mathcal{H} = \frac{\rho(w(z))}{\sigma(z)} \lvert w_{z}\rvert^2 
\end{equation}
\begin{equation*}
\mathcal{L} = \frac{\rho(w(z))}{\sigma(z)} \lvert w_{\bar{z}}\rvert^2 
\end{equation*}
which implies that the Hopf differnetial $\Phi$ of $w$ satisfies:
\begin{equation}\label{phi2}
\lvert \Phi \rvert^2 = \mathcal{H}\mathcal{L}/\sigma^2
\end{equation}

Also, we find that the Euler-Lagrange equation \eqref{har} yields the Bochner equation:
\begin{equation}
\Delta h = 2e^h - 2 \lVert \Phi\rVert^2 e^{-h} -2 
\end{equation}
where $h = \ln \mathcal{H}$.

Denote $\Phi_i = \Phi(Y_i)$ for $i=1,2$, 
and suppose $\Phi_1 = \Phi_2 = \Phi$.

Then it follows that $h_1 = h_2$ by an application of the Maximum Principle: namely, if $h_1>h_2$ somewhere, then at a maximum of the difference $h_1-h_2$ we would have:
$$0 \geq \Delta (h_1-h_2) = (e^{h_1} - e^{h_2}) - \lvert \Phi^2 \rvert^2 (e^{-h_1} - e^{-h_2})  >0$$
which is a contradiction, and a similar contradiction is reached when $h_1<h_2$ somewhere.

This implies that $\mathcal{H}_1 = \mathcal{H}_2$ and by \eqref{phi2} we have $\mathcal{L}_1=\mathcal{L}_2$ also. Hence the energy densities $e_1= e_2$, which implies that the pullback tensors $w_1^\ast \rho_1 = w_2^\ast \rho_2$.  Thus, the composition  $w_1 \circ w_2^{-1}: Y_2 \to Y_1$ is an isometry isotopic to the identity, so we can conclude that $Y_1=Y_2$. $\qed$ \\

This establishes that  the map $\hat \Psi$ in \eqref{hatpsi}  is proper and injective. Continuity of $\hat \Psi$ follows from the injectivity, and is also noted in the literature (for example in \cite{E-E}). Hence the Invariance of Domain applies, we conclude that  $\hat{\Psi}$  is a homeomorphism.

\subsection*{Properness}

For $i\geq 1$, onsider a sequence $\Psi(Y_i) = q_i\in \Q(X,D,P)$ with $Y_i \in \T(P)$ such that the image quadratic differential lies in a compact set $K_0$. 
Let 
\begin{equation}\label{hi-seq}
h_i:X\setminus p \to Y_i
\end{equation}
 be the associated sequence of harmonic diffeomorphisms to the crowned hyperbolic surfaces $Y_i \in \T(P)$.
 
To prove the properness of $\Psi$, we need to establish that the sequence of crowned hyperbolic surfaces $Y_i$ sub-converge to a crowned hyperbolic surface $Y \in \T(P)$. 

That is,

\begin{prop}\label{last-prop}






The sequence of hyperbolic surfaces $\{Y_i\}_{i\geq 1}$  in \eqref{hi-seq} sub-converge to a crowned hyperbolic surface $Y \in \T(P)$.

\end{prop}

Note that all simple closed geodesics in $X \setminus p$, with respect to the uniformizing (cusped) hyperbolic metric, lie in a compact subsurface $S$. In what follows we fix this subsurface $S$ carrying  the topology of the punctured surface. Moreover, we shall assume that $\partial S$ is homotopic to the puncture.  (For example, $S$ can be taken to be $X_m$ for some large $m$ in the compact exhaustion considered previously.)\\

Adapting an argument of Wolf (Lemma 3.2 of \cite{Wolf0}) one can show:

\begin{lem} The sequence of harmonic maps $\{h_i\}_{i\geq 1}$ above (\textit{cf.} \eqref{hi-seq}) has uniformly bounded energy when restricted to the subsurface $S$. 
\end{lem}
\begin{proof}
Suppose not, that is, assume $\mathcal{E}(h_i\vert_S) \to \infty$ as $i\to \infty$.

For a harmonic map $h$ on $S$ we have that the Jacobian  $\mathcal{J} = \mathcal{H} - \mathcal{L}>0$  where $\mathcal{H}, \mathcal{L}$ are the holomorphic and anti-holomorphic energies introduced earlier (see \eqref{hl}). 
Thus we have:
\begin{equation*}
\int\limits_S \mathcal{H} \sigma dz d\bar{z}  + 2\pi \chi(S)  = \int\limits_S \mathcal{L} \sigma dz d\bar{z}  \leq  \int\limits_S \lvert \phi \rvert \sigma^2 dz d\bar{z}
\end{equation*}
where we have used \eqref{phi2}.

Adding the terms on either side of the first equality, and using the fact that the energy density $e= \mathcal{H} + \mathcal{L}$ (see \eqref{energy}), we then obtain that 
\begin{equation*}
\mathcal{E}(h) = \int\limits_ S e \sigma dz d\bar{z}  + 2\pi \chi(S)   \leq 2 \int\limits_S \lvert \phi \rvert \sigma^2 dz d\bar{z} =: \lVert \phi\rVert_S
\end{equation*}

Applying this inequality to the sequence of maps $h_i$, we obtain that the norms $\lVert q_i\rVert_S \to \infty$ as $i\to \infty$, which contradicts the fact that $q_i$ lies in a compact subset $K_0$ of $\Q(X,D,P)$.
\end{proof}

\bigskip

As  a consequence, we can show:

\begin{proof}[Proof of Proposition \ref{last-prop}]

 By the previous lemma, there is a uniform bound on the energy of ${h}_i$ on the fixed subsurface $S$. Applying the Courant-Lebesgue Lemma,  there is a constant $C_1>0$ (independent of $i$) such that  the length $l_{Y_i}(h_i(\gamma)) \leq C_1$ for any simple closed geodesic $\gamma$ on $X\setminus p$ with respect to a cusped hyperbolic metric. 
(Note that by the definition of $S$, such a geodesic must be contained in the subsurface $S$.) 





Note that the uniform length bound also holds for images of the boundary curve $\partial S$, which are homotopic to the geodesic boundaries of the hyperbolic crown $\CC_i$ on $Y_i$.

Recall that a  marked hyperbolic surface with geodesic boundary is determined up to isometry by the lengths of a finite collection of simple closed curves including the boundary.

Thus the (un)crowned surfaces  $Y_i \setminus \CC_i \subset S $ sub-converge to a hyperbolic surface with geodesic boundary $\bar{Y}$ as $i\to \infty$. 

For the sub-convergence of the crowns $\CC_i$ on $Y_i$, we use the fact that $q_i$ lie in a compact subset of $\Q(X_i, D, P)$ for all $i\geq 1$, and in particular, $q_i\vert_U$ will sub-converge to a meromorphic quadratic differential $q_\infty$ on $U \cong \D$ with a pole of order $n$ at the origin.
The restriction of the harmonic maps $h_i$ to $U$ thus sub-converge to a harmonic map to the limiting hyperbolic crown $\CC$ with Hopf differential $q_\infty$ (\textit{cf.} Theorem \ref{asm}). 

This sub-convergence includes the boundary twist parameters which control the boundary identifications of the crown $\CC_i$  and surface $Y_i \setminus \CC_i$.  

Together, we obtain a subsequence of $Y_i$ converging to a crowned hyperbolic surface $Y$ as claimed.
\end{proof}

Thus $\Psi$ is proper, injective and continuous, and hence by the Invariance of Domain, is  homeomorphism, completing the proof of Theorem \ref{thm2-case}. As before, the argument extends to case of more than one marked point, completing the proof of Theorem \ref{thm2}.

\bibliographystyle{amsalpha}
\bibliography{Crown-refs}
 
\end{document}